\documentclass[preprint]{elsarticle}

\usepackage{amsmath}
\usepackage{amsfonts}
\usepackage{amssymb}
\usepackage{amsthm}

\newtheorem{remark}{\it Remark}[section]

\newtheorem{proposition}{\it Proposition}[section]
\newtheorem{theorem}{\it Theorem}[section]

\newtheorem{corollary}{\it Corollary}[section]

\newcommand{\tra}[1]{\,{\vphantom{#1}}^{\textsc{t}}\!{#1}}

\begin{document}

\title{Quadratic choreographies}

\author[LMPA]{P.~Ryckelynck\corref{cor1}}
\ead{ryckelyn@lmpa.univ-littoral.fr}
\author[LMPA]{L.~Smoch}
\ead{smoch@lmpa.univ-littoral.fr}

\address[LMPA]{ULCO, LMPA, F-62100 Calais, France\\Univ Lille Nord de France, F-59000 Lille, France. CNRS, FR 2956, France.}

\begin{abstract}
This paper addresses the classical and discrete Euler-Lagrange equations for systems of $n$ particles interacting quadratically in $\mathbb{R}^d$. By highlighting the role played by the center of mass of the particles, we solve the previous systems via the classical quadratic eigenvalue problem (QEP) and its discrete transcendental generalization. The roots of classical and discrete QEP being given, we state some conditional convergence results. Next, we focus especially on periodic and choreographic solutions and we provide some numerical experiments which confirm the convergence.
\end{abstract}


\maketitle

\section{Introduction}
This paper seeks to continue the development of the theory for the discrete calculus of variations which was initiated by Cresson and al., see \cite{Cre,CFT}. It consists originally in replacing the derivative $\dot{\mathbf{x}}(t)$ of the dynamic variable $\mathbf{x}(t)$ defined on $[t_0,t_f]$ with a $2N+1$ terms scale derivative
\begin{equation}
\Box _\varepsilon \mathbf{x}(t)=\sum_{j=-N}^N\frac{\gamma_j}{\varepsilon}\mathbf{x}(t+j\varepsilon)\chi _{-j}(t),~~t\in[t_0,t_f],~~\gamma_j\in\mathbb{C}
\label{ourbox}
\end{equation}
where $\chi_j(t)$ denotes the characteristic function of $[t_0,t_f]\cap[t_0+j\varepsilon,t_f+j\varepsilon]$, for some time delay $\varepsilon$. 

We consider a lagrangian $\mathcal{L}$ of $n$ particles in $\mathbb{R}^d$, where $d$ denotes the ``physical" dimension. The principle of least action may be extended to the case of non-differentiable dynamic variables. For conservative systems, the equations of motion may be returned as the following two dynamic sets of equations
\begin{equation}
\displaystyle \ddot{\mathbf{x}}_j(t)=F_j(\mathbf{x}_1,\ldots,\mathbf{x}_n)\mbox{ and }\displaystyle -\Box_{-\varepsilon}\Box_{\varepsilon}\mathbf{x}_j(t)=\tilde{F}_j(\mathbf{x}_1,\ldots,\mathbf{x}_n)
\label{systdyn}
\end{equation}
where the functions $F_j,\tilde{F}_j$ are built from the specific interaction between the particles. While the first system in (\ref{systdyn}) deals with ODE, the second one consists in a set of functional difference equations. We investigate for each system the existence of pseudo-periodic solutions of the shape
\begin{equation}
\mathbf{u}(t)=\mathbf{u}_0+\sum_{\ell=1}^{K}e^{\lambda_\ell t}\mathbf{u}_{\ell},
\label{psdf}
\end{equation}
where $\mathbf{u}_0$ and $\mathbf{u}_{\ell}$ constitute a family of $K+1$ vectors of $\mathbb{C}^d$ and $(\lambda_\ell)_\ell\in(\mathbb{C}^\star)^K$ is a sequence of $K$ distinct complex numbers.

The rest of this paper is organized as follows. Section 2 is devoted to the derivation of the classical and discrete Euler-Lagrange equations (respectively abbreviated as C.E.L. and D.E.L.) and highlights the role played by the center of mass  $\frac{1}{n}\mathbf{x}_s(t)$ where $\mathbf{x}_s(t)=\sum_j\mathbf{x}_j(t)$. 
In Section 3, we present a method for solving the equations of motion for generic lagrangians for C.E.L. as well as D.E.L.. The first step of this method determines $\mathbf{x}_s(t)$ from some generalized (quadratic or transcendental) eigenvalue problem. The second step seeks $\mathbf{x}_j(t)$ from $\mathbf{x}_s(t)$ by solving another eigenvalue problem. Section 4 is devoted to the convergence of the generalized eigenvalue problem linked to the D.E.L. as $\varepsilon$ tends to 0. Section 5 deals with the existence and the features of periodic and choreographic solutions. Finally, Section 6 presents some numerical experiments illustrating the phenomenon of convergence as $\varepsilon$ tends to 0, for some various operators $\Box_\varepsilon$.

\section{Equations of motion for symmetric quadratic lagrangians of $n$ particles systems in $\mathbb{R}^d$}

The principle of least discrete action has been developed in \cite{RS1,RS2} to which we refer throughout the paper.  We denote by $\mathcal{C}_{pw}$ the space of the functions $\mathbf{x}:[t_0,t_f]\rightarrow \mathbb{R}^d$ continuous on each interval $[t_0+j\varepsilon,t_0+(j+1)\varepsilon ]\cap [t_0,t_f]$ for all $j\geq 0$ and small enough, i.e. $j\leq\frac{t_f-t_0}{\varepsilon}$. If $\mathbf{X}=(\mathbf{x}_1,\ldots\mathbf{x}_n)$ denotes a system of $n$ functions in $\mathcal{C}_{pw}$, we may think of $\mathbf{X}$ as the set of dynamic variables describing the state of a system of $n$ interacting particles in $\mathbb{R}^d$.

We consider actions $\mathcal{A}_{cont},\mathcal{A}_{disc}:\mathcal{C}_{pw}^n\rightarrow\mathbb{R}$ of the shape 
\begin{equation}
\mathcal{A}_{cont}(\mathbf{X})=\int_{t_0}^{t_f} \mathcal{L}(\mathbf{X}(t),\dot{\mathbf{X}}(t))dt,\hspace{0.5cm}\mathcal{A}_{disc}(\mathbf{X})=\int_{t_0}^{t_f} \mathcal{L}(\mathbf{X}(t),\Box_\varepsilon{\mathbf{X}}(t))dt.
\label{action}
\end{equation}
From now on, we drop $t$ from the formulas when it is clear enough.

We introduce a general quadratic lagrangian of $n$ particles in $\mathbb{R}^d$, compatible with discrete symmetries of the system. Let $J_1,\ldots,J_5\in\mathcal{C}^0([t_0,t_f],\mathbb{R}^{d\times d})$, $J_1,\ldots,J_4$ be symmetric matrices and $J_6,J_7\in\mathcal{C}^0([t_0,t_f],\mathbb{R}^d)$. For an isolated particle with position $\mathbf{x}$ and velocity $\mathbf{y}$ we may set
\begin{center}
$\displaystyle \mathcal{L}_1(\mathbf{x},\mathbf{y})=\frac{1}{2}\tra{\mathbf{y}}J_1\mathbf{y}+\frac{1}{2}\tra{\mathbf{x}}J_2\mathbf{x}+\tra{\mathbf{x}}J_5\mathbf{y}+\tra{J}_6\mathbf{y}+\tra{J}_7\mathbf{x}$.
\end{center}
Next, two particles with positions $\mathbf{x}_j$, $\mathbf{x}_k$, and velocities $\mathbf{y}_j$, $\mathbf{y}_k$ are interacting for pairs in conformity with the following lagrangian
\begin{center}
$\displaystyle \mathcal{L}_2(\mathbf{x}_j,\mathbf{y}_j,\mathbf{x}_k,\mathbf{y}_k)=\tra{\mathbf{y}}_j J_3\mathbf{y}_k+\tra{\mathbf{x}}_j J_4\mathbf{x}_k$.
\end{center}
Therefore, the lagrangian of the whole system is 
\begin{equation}
\mathcal{L}=\sum_{j=1}^n\mathcal{L}_1(\mathbf{x}_j,\dot{\mathbf{x}}_j)+\sum_{j\neq k}\mathcal{L}_2(\mathbf{x}_j,\dot{\mathbf{x}}_j,\mathbf{x}_k,\dot{\mathbf{x}}_k).
\label{choreolag}
\end{equation}

\begin{theorem}
Let $\mathbf{X}=(\mathbf{x}_1,\ldots,\mathbf{x}_n)$ in $\mathcal{C}_{pw}^n$ and $\displaystyle \mathbf{x}_s=\sum_{j=1}^n\mathbf{x}_j$.\\
A necessary and sufficient condition for $\mathbf{X}$ to be a critical point of $\mathcal{A}_{cont}$ in $\mathcal{C}_{pw}^n$ is that $\mathbf{X}$ satisfies the dynamic system
\begin{eqnarray}
(J_1-2J_3)\ddot{\mathbf{x}_{j}}+(-2\dot{J_3}+\tra{J}_5+\dot{J_1}-J_5)\dot{\mathbf{x}}_{j}+
(\tra{\dot{J}}_5-J_2+2J_4)\mathbf{x}_j=\nonumber\\
-2J_3\ddot{\mathbf{x}}_s-2\dot{J_3}\dot{\mathbf{x}}_s+2J_4\mathbf{x}_s+(J_7-\dot{J_6}),~~\forall j.
\label{CELchoreo}
\end{eqnarray}
A necessary and sufficient condition for $\mathbf{X}$ to be a critical point of $\mathcal{A}_{disc}$ in $\mathcal{C}_{pw}^n$ is that $\mathbf{X}$ satisfies the linear functional recurrence system of equations
\begin{eqnarray}
\Box_{-\varepsilon}((J_1-2J_3)\Box_\varepsilon\mathbf{x}_j)+\Box_{-\varepsilon}(\tra{J}_5\mathbf{x}_j)+J_5\Box_\varepsilon\mathbf{x}_j+(J_2-2J_4)\mathbf{x}_j=\nonumber\\
-2\Box_{-\varepsilon}(J_3\Box_\varepsilon\mathbf{x}_s)-2J_4\mathbf{x}_s-(\Box_{-\varepsilon}J_6+J_7)),~~\forall j.
\label{DELchoreo}
\end{eqnarray}
\label{thmeqofmot}
\end{theorem}

\begin{proof}
We first recall the classical and discrete Euler-Lagrange equations, which are respectively given by 
\begin{equation}
-\frac{d}{dt}\frac{\partial \mathcal{L}}{\partial \dot{\mathbf{x}}_j}(t,\mathbf{X}(t),\dot{\mathbf{X}}(t))+\frac{\partial \mathcal{L}}{\partial \mathbf{x}_j}(t,\mathbf{X}(t),\dot{\mathbf{X}}(t))=0
\label{mcel}
\end{equation}
and
\begin{equation}
\Box_{-\varepsilon }\frac{\partial \mathcal{L}}{\partial \Box_{\varepsilon}\mathbf{x}_j}(t,%
\mathbf{X}(t),\Box _\varepsilon \mathbf{X}(t))+\frac{\partial \mathcal{L}}{%
\partial \mathbf{x}_j}(t,\mathbf{X}(t),\Box _\varepsilon \mathbf{X}(t))=0,
\label{mcelbis}
\end{equation}
for all $j\in\{1,\ldots,n\}$.\\
The computation of gradients of $\mathcal{L}$ needs the following property : if $\mathbf{a},\mathbf{b}\in\mathbb{R}^d$,
\begin{center}
$\nabla_\mathbf{x}(\tra{\mathbf{a}}\mathbf{x}+\tra{\mathbf{x}}\mathbf{b})=\mathbf{a}+\mathbf{b}$.
\end{center}
Let us prove (\ref{CELchoreo}).
Because of the symmetry of $J_1,J_2,J_3,J_4$, we get
\begin{center}
$\begin{array}{rcl}
\displaystyle \frac{\partial \mathcal{L}}{\partial\dot{\mathbf{x}}_{j}} & = & \displaystyle J_1\dot{\mathbf{x}}_{j}+2J_3\sum_{k\neq j}\dot{\mathbf{x}}_k+\tra{J}_5\mathbf{x}_j+J_6\mbox{ and}\\
\displaystyle \frac{\partial \mathcal{L}}{\partial\mathbf{x}_j} & = & \displaystyle J_2\mathbf{x}_j+2J_4\sum_{k\neq j}\mathbf{x}_k+J_5\dot{\mathbf{x}}_{j}+J_7.
\end{array}$
\end{center}
Then, by setting $\displaystyle \mathbf{x}_s=\sum_{j=1}^n\mathbf{x}_j$, we get
\begin{center}
$\begin{array}{rcl}
\displaystyle \frac{\partial \mathcal{L}}{\partial\dot{\mathbf{x}}_{j}} & = & \displaystyle(J_1-2J_3)\dot{\mathbf{x}}_{j}+2J_3\dot{\mathbf{x}}_s+\tra{J}_5\mathbf{x}_j+J_6\mbox{ and}\vspace{0.1cm}\\ 
\displaystyle \frac{\partial \mathcal{L}}{\partial\mathbf{x}_j} & = & (J_2-2J_4)\mathbf{x}_j+2J_4\mathbf{x}_s+J_5\dot{\mathbf{x}}_{j}+J_7.
\end{array}$
\end{center}
The equation (\ref{mcel}) gives for all $j$ :
\begin{center}
$\displaystyle (\dot{J_1}-2\dot{J_3})\dot{\mathbf{x}}_{j}+(J_1-2J_3)\ddot{\mathbf{x}_{j}}+\tra{\dot{J}}_5\mathbf{x}_j+\tra{J}_5\dot{\mathbf{x}}_{j}+\dot{J_6}+2\dot{J_3}\dot{\mathbf{x}}_s+2J_3\ddot{\mathbf{x}_s}=(J_2-2J_4)\mathbf{x}_j+J_5\dot{\mathbf{x}}_{j}+J_7+2J_4\mathbf{x}_s$
\end{center}
which is equivalent to (\ref{CELchoreo}).\\
The proof of (\ref{DELchoreo}) is quite similar since (\ref{mcelbis}) gives for all $j$ :
\begin{center}
$\displaystyle \Box_{-\varepsilon}(J_1\Box_\varepsilon\mathbf{x}_j-2J_3\Box_\varepsilon\mathbf{x}_j)+\Box_{-\varepsilon}(\tra{J}_5\mathbf{x}_j)+J_5\Box_\varepsilon\mathbf{x}_j+(J_2-2J_4)\mathbf{x}_j=-2\Box_{-\varepsilon}(J_3\Box_\varepsilon\mathbf{x}_s)-2J_4\mathbf{x}_s-(\Box_{-\varepsilon}J_6+J_7)$.
\end{center}
which implies (\ref{DELchoreo}).
\end{proof}
We notice that the equations (\ref{CELchoreo}) and (\ref{DELchoreo}) are quite uncoupled since the coupling is realized only through the vector $\mathbf{x}_s$. We mention two simple consequences of the previous result. The first one arises from summing all equations (\ref{CELchoreo}) or summing all equations in (\ref{DELchoreo}) over $j$, and the second one deals with time-independent lagrangians such that $J_5$ is skew-symmetric.

\begin{corollary}
If $\mathbf{X}$ is a solution to (\ref{CELchoreo}), then the sum $\mathbf{x}_s$ satisfies the dynamic system
\begin{eqnarray}
(J_1+2(n-1)J_3)\ddot{\mathbf{x}_s}+(\dot{J_1}+2(n-1)\dot{J_3}+\tra{J}_5-J_5)\dot{\mathbf{x}}_s-\nonumber\\
(J_2+2(n-1)J_4-\tra{\dot{J}}_5)\mathbf{x}_s+n(\dot{J_6}-J_7)=0.
\label{CELchoreoCTsum}
\end{eqnarray}
Similarly, if $\mathbf{X}$ is a solution to (\ref{DELchoreo}), then $\mathbf{x}_s$ satisfies the functional equation
\begin{eqnarray}
\Box_{-\varepsilon}(J_1\Box_\varepsilon\mathbf{x}_s)+2(n-1)\Box_{-\varepsilon}(J_3\Box_\varepsilon\mathbf{x}_s)+\Box_{-\varepsilon}(\tra{J}_5\mathbf{x}_s)+\nonumber\\
J_5\Box_\varepsilon\mathbf{x}_s+(J_2+2(n-1)J_4)\mathbf{x}_s+n\Box_{-\varepsilon}J_6+nJ_7=0.
\label{DELchoreoCTsum}
\end{eqnarray}
\label{coro2.1}
\end{corollary}

\begin{corollary}
Suppose that the functions $J_k(t)$, $k\in\{1,\ldots,7\}$, are constant w.r.t. time and $J_5$ is skew-symmetric. The systems of equations (\ref{CELchoreoCTsum}) and (\ref{CELchoreo}) simplify respectively into
\begin{equation}
(J_1+2(n-1)J_3)\ddot{\mathbf{x}_s}-2J_5\dot{\mathbf{x}}_s-(J_2+2(n-1)J_4)\mathbf{x}_s=nJ_7
\label{CELchoreoCTsumsimple}
\end{equation}
and
\begin{equation}
(J_1-2J_3)\ddot{\mathbf{x}_{j}}-2J_5\dot{\mathbf{x}}_{j}-(J_2-2J_4)\mathbf{x}_j=-2J_3\dot{\mathbf{x}}_s+2J_4\mathbf{x}_s+J_7.
\label{CELchoreoCT}
\end{equation}
Similarly, the systems of functional recurrence equations (\ref{DELchoreoCTsum}) and (\ref{DELchoreo}) simplify respectively into
\begin{eqnarray}
(J_1+2(n-1)J_3)\Box_{-\varepsilon}\Box_\varepsilon\mathbf{x}_s+J_5(\Box_\varepsilon\mathbf{x}_s-\Box_{-\varepsilon}\mathbf{x}_s)+(J_2+2(n-1)J_4)\mathbf{x}_s & = & \nonumber\\
-n\Box_{-\varepsilon}(1)J_6-nJ_7 \qquad\qquad\qquad\qquad& & 
\label{DELchoreoCTsumsimple}
\end{eqnarray}
and
\begin{eqnarray}
(J_1-2J_3)\Box_{-\varepsilon}\Box_\varepsilon\mathbf{x}_j+J_5(\Box_\varepsilon\mathbf{x}_j-\Box_{-\varepsilon}\mathbf{x}_j)+(J_2-2J_4)\mathbf{x}_j & = & \nonumber\\
-2J_3\Box_{-\varepsilon}\Box_\varepsilon\mathbf{x}_s-2J_4\mathbf{x}_s+\Box_{-\varepsilon}J_6-J_7.\qquad & &
\label{DELchoreoCT}
\end{eqnarray}
\label{coro2.2}
\end{corollary}

\begin{remark}\rm
Let $J_8,J_9,J_{10}\in\mathcal{C}^0([t_0,t_f],\mathbb{R}^{nd\times nd})$ denote the matrices constructed by blocks as follows
\begin{center}
$J_8=\begin{pmatrix}J_1 & 2J_3 & \ldots & 2J_3\\ 2J_3 & J_1 & \ddots & \vdots\\\vdots & \ddots & \ddots & 2J_3\\2J_3 & \ldots & 2J_3 & J_1\end{pmatrix}$, $J_9=\begin{pmatrix}J_2 & 2J_4 & \ldots & 2J_4\\ 2J_4 & J_2 & \ddots & \vdots\\\vdots & \ddots & \ddots & 2J_4\\2J_4 & \ldots & 2J_4 & J_2\end{pmatrix}$,
\end{center}
and $J_{10}=diag(J_5,\ldots,J_5)$. If $\mathbf{X}\in\mathcal{C}_{pw}^n$ is a critical point of $\mathcal{A}_{cont}$, i.e. $\mathbf{X}$ satisfies (\ref{CELchoreo}) and if $\mathbf{Y}\in\mathcal{C}_{pw}^n$ vanishes at $t_0$ and $t_f$, then we have 
\begin{center}
$\displaystyle \mathcal{A}_{cont}(\mathbf{X}+\mathbf{Y})-\mathcal{A}_{cont}(\mathbf{X})=\frac{1}{2}\int_{t_0}^{t_f}(\tra{\dot{\mathbf{Y}}}J_8\dot{\mathbf{Y}}+\tra{\mathbf{Y}}J_9\mathbf{Y}+\tra{\mathbf{Y}}J_{10}\dot{\mathbf{Y}})dt$. 
\end{center}
As a consequence, if the integrand is a positive definite quadratic form w.r.t. $(\mathbf{Y},\dot{\mathbf{Y}})$, then the equations (\ref{CELchoreo}) are necessary and sufficient conditions for a strict minimum of the action $\mathcal{A}_{cont}$ to occur. Especially, if $J_5=0$ and the matrices $J_8$ and $J_9$ are definite positive, such an optimum occurs.
\end{remark}

\begin{remark}\rm
Let $A\in\mathbb{R}^{d\times d}$ a nonsingular matrix and $\mathbf{b}\in\mathbb{R}^d$ be given. Let us consider the transformation of the whole system
\begin{center}
$\hat{\mathbf{x}}_j(t)=A\mathbf{x}_j(t)+\mathbf{b}$.
\end{center}
Then this transformation is covariant for quadratic lagrangians in the sense that $\mathcal{L}(\mathbf{x}_1,\ldots,\mathbf{x}_n)$ is of the shape (\ref{choreolag}) iff $\hat{\mathcal{L}}(\hat{\mathbf{x}}_1,\ldots,\hat{\mathbf{x}}_n)$ is of the shape (\ref{choreolag}). Moreover the properties of symmetry for $\hat{J}_1,\ldots,\hat{J}_5$ are equivalent to those for $J_1,\ldots,J_5$. At last, the equations of motion (\ref{systdyn}) are covariant altogether as can be shown from the formula for affine forces 
\begin{center}
$\hat{F}_j(\hat{\mathbf{X}})=AF_j(A^{-1}(\hat{\mathbf{x}}_1-\mathbf{b}),\ldots,A^{-1}(\hat{\mathbf{x}}_n-\mathbf{b}))$.
\end{center}
where $\hat{\mathbf{X}}=(\hat{\mathbf{x}}_1,\ldots,\hat{\mathbf{x}}_n)$.
\end{remark}

\begin{remark}\rm
If $J_5=0$, the system has a lagrangian of the shape $T(\dot{\mathbf{X}})-U(\mathbf{X})$ and consequently, it is conservative, i.e. the energy
\begin{equation}
\sum_{j=1}^n\left(\frac{1}{2}\tra{\dot{\mathbf{x}}}_{j}J_1\dot{\mathbf{x}}_{j}-
\frac{1}{2}\tra{\mathbf{x}}_j J_2\mathbf{x}_j+
\tra{J}_6\dot{\mathbf{x}}_{j}-
\tra{J}_7\mathbf{x}_j\right)+\sum_{j\neq k}\left(\tra{\dot{\mathbf{x}}}_{j}J_3\dot{\mathbf{x}}_k-
\tra{\mathbf{x}}_j J_4\mathbf{x}_k\right)
\label{comenergy}
\end{equation}
is a constant of motion.
\end{remark}

\section{Solutions of equations of motion in the general case}

\subsection{Preliminaries on Quadratic Eigenvalue Problems}

We provide in this section the solutions to problems presented in Corollaries \ref{coro2.1} and \ref{coro2.2}. From now on we suppose that the vectors and matrices $J_k$, $k=1,\ldots,7$, are time-independent and that $J_5$ is skew-symmetric. By general case, we mean that the set of coefficients $(J_k)_{k=1,\ldots,7}$, satisfying the conditions (\ref{generalconditionsforQEPCEL}) and (\ref{generalconditionsforQEPDEL}) below is everywhere dense in $(\mathbb{R}^{d\times d})^5\times(\mathbb{R}^d)^2$. 

According to \cite{L1,L2,TM} we define the Quadratic Eigenvalue Problem associated to $(A,B,C)$ as the search of the complex roots of the discriminantal equation 
\begin{equation}
\det(A\lambda^2+B\lambda+C)=0
\end{equation}
where the l.h.s. is a polynomial of $\lambda$ of degree $2d$, together with the description of the various kernels $\ker(A\lambda^2+B\lambda+C)$. The reader is referred to \cite{TM} for a survey of theory applications and algorithms of the QEP.

It is a classical old fact \cite{L1,GLR,TM} that, if all the roots $\lambda_\ell$, $\ell\in\{1,\ldots,2d\}$, of $A\lambda^2+B\lambda+C$ are distinct  and $C$ is invertible, the general solution to the dynamic system $A\ddot{\mathbf{x}}+B\dot{\mathbf{x}}+C\mathbf{x}=\mathbf{k}_0$ has the shape (\ref{psdf}), with $K=2d$, $\mathbf{u}_0=C^{-1}\mathbf{k}_0$ and $\mathbf{u}_{\ell}\in\ker(\lambda_\ell^2A+\lambda_\ell B+C)$. When the number of roots is less than $2d$, slight more complicated expressions for the solutions may be found in \cite{L1,L2,TM}\\
Because of equations (\ref{CELchoreoCTsum}) to (\ref{DELchoreoCT}) and the previous discussion, we may provide, under specific assumptions, the shape of the solutions to C.E.L. and D.E.L..

\subsection{The case of C.E.L.}

We introduce the matrix-valued function
\begin{equation}
\mathcal{P}_\nu(\lambda):=(J_1+2(\nu-1)J_3)\lambda^2-2J_5\lambda -(J_2+2(\nu-1)J_4)
\label{Jlambda}
\end{equation}
and the following subsets of $\mathbb{C}$ 
\begin{center}
$\mathcal{Q}_\nu:=\left\{\lambda\in\mathbb{C}/\det(\mathcal{P}_\nu(\lambda))=0\right\}$, $\forall \nu\in\mathbb{R}$.
\end{center}

\begin{proposition}
We assume that 
\begin{eqnarray}
|\mathcal{Q}_n|=|\mathcal{Q}_0|=2d,~\mathcal{Q}_n\cap \mathcal{Q}_0=\emptyset,\nonumber\\
~\det(\mathcal{P}_n(0))\neq 0\mbox{ and }\det(\mathcal{P}_0(0))\neq 0\label{generalconditionsforQEPCEL}
\end{eqnarray}
where $n$ denotes the number of interacting particles. Then all the solutions $\mathbf{x}_s$ and $\mathbf{x}_j$ to (\ref{CELchoreoCTsumsimple}) and (\ref{CELchoreoCT}) respectively are of the shape (\ref{psdf}) with $K=2d$.
\label{prop3.1}
\end{proposition}

\begin{proof}
Since $\det(J_2+2(n-1)J_4)\neq 0$, we may define 
\begin{center}
$\mathbf{x}_{s,0}=-n(J_2+2(n-1)J_4)^{-1}J_7$
\end{center} 
and we have obviously $0\notin \mathcal{Q}_n$. The first condition $|\mathcal{Q}_n|=2d$ guarantees that the solution $\mathbf{x}_s$ to (\ref{CELchoreoCTsumsimple}) is of the shape (\ref{psdf}) with $K=2d$, i.e. 
\begin{equation}
\displaystyle \mathbf{x}_s(t)=\mathbf{x}_{s,0}+\sum_{\alpha\in \mathcal{Q}_n}e^{\alpha t}\mathbf{x}_{s,\alpha},
\label{solgenxsCEL}
\end{equation}
for some convenient vectors $\mathbf{x}_{s,\alpha}\in\ker\left(\mathcal{P}_n(\alpha)\right)$. Since $\mathcal{Q}_n\cap \mathcal{Q}_0=\emptyset$, the matrix $\mathcal{P}_0(\alpha)$ is invertible for each $\alpha\in \mathcal{Q}_n$ so we may define for $j\in\{1,\ldots,n\}$
\begin{center}
$\mathbf{x}_{j,\alpha}=2\mathcal{P}_0(\alpha)^{-1}(J_4-\alpha^2J_3)\mathbf{x}_{s,\alpha}$, $\forall \alpha\in \mathcal{Q}_n$.
\end{center}
Let us set $\mathbf{x}_{j,0}=-(J_2-2J_4)^{-1}(2J_4\mathbf{x}_{s,0}+J_7)$. Straightforward computations show that $\mathbf{x}_{j,0}=\frac{1}{n}\mathbf{x}_{s,0}$ and $\mathbf{x}_{j,\alpha}=\frac{1}{n}\mathbf{x}_{s,\alpha}$, $\forall \alpha\in \mathcal{Q}_n$. Therefore, we have proved that $\mathbf{x}_{j,0}+\sum_{\alpha\in \mathcal{Q}_n}e^{\alpha t}\mathbf{x}_{j,\alpha}=\frac{1}{n}\mathbf{x}_s(t)$ is a particular solution to (\ref{CELchoreoCT}). Hence, the general solution to (\ref{CELchoreoCT}) is given by the formula (\ref{psdf}) with $K=2d$, i.e.
\begin{equation}
\displaystyle \mathbf{x}_{j}(t)=\frac{1}{n}\mathbf{x}_s(t)+\sum_{\beta\in \mathcal{Q}_0}e^{\beta t}\mathbf{x}'_{j,\beta}
\label{solgenxlCEL}
\end{equation}
for some convenient vectors $\mathbf{x}'_{j,\beta}\in\ker\left(\mathcal{P}_0(\beta)\right)$.
\end{proof}

Let us consider first the well-posedness of the Dirichlet problem for C.E.L., i.e. (\ref{CELchoreoCTsumsimple}) and (\ref{CELchoreoCT}). We use some facts mentioned in \cite[Section 3]{TM} which are consequences of the existence of the Smith form for regular QEP, see also \cite{L1,GLR}. Let $\nu=0$ or $\nu=n$. Since $\mathcal{P}_\nu(\lambda)$ admits exactly $2d$ distinct roots, then $\dim\ker(\mathcal{P}_\nu(\lambda))=1$, $\forall\lambda\in \mathcal{Q}_\nu$ and the union of the various $\ker(\mathcal{P}_\nu(\lambda))$ spans $\mathbb{C}^d$. If $\nu=n$, we may decompose $\mathbf{x}_s(t_0)$ and $\mathbf{x}_s(t_f)$ on the family $\{\mathbf{x}_{s,\alpha}\}_{\alpha\in\mathcal{Q}_n}$ and we obtain a linear system of $2d$ equations w.r.t the $2d$ unknowns which are abscissas of $\mathbf{x}_{s,\alpha}$ along the linear straight lines $\ker(\mathcal{P}_n(\alpha))$, $\alpha\in \mathcal{Q}_n$. 
We proceed in a similar way when $\nu=0$ and $j\in\{1,\ldots,n\}$. We get a linear system of $2d$ equations by decomposing $\mathbf{x}_j(t_0)$ and $\mathbf{x}_j(t_f)$ w.r.t. the $2d$ unknowns which are the abscissas of  $\mathbf{x}'_{j,\beta}$, $\beta\in \mathcal{Q}_0$. Hence, the Dirichlet problem amounts to solving $n+1$ uncoupled square systems of size $2d$ (the very last one being useless due to the definition of $\mathbf{x}_s$). Each of the previous system is Cramer for almost all couple $(t_0,t_f)$. Indeed, the determinant of each system has the shape $P(e^{\lambda t_0},e^{\lambda t_f})_{\lambda\in \mathcal{Q}_\nu}$, $\nu=0$ or $\nu=n$, $P$ being a polynomial with coefficients depending on the coordinates of $\mathbf{x}_{s,\lambda}$ and $\mathbf{x}_{j,\lambda}$.

\subsection{The case of D.E.L.}

Let us extend the Proposition \ref{prop3.1} to the D.E.L. case. It should be emphasized here that D.E.L. do not admit in general a unique solution. Nevertheless, given a solution, there exists one and only one pseudo-periodic solution which agrees with the first one on a grid $\mathcal{G}_\varepsilon$. As well as the study of autonomous dynamic differential systems leads to QEP, the study of autonomous difference equations leads to transcendental eigenvalue problem associated to the following complicated matrix
\begin{eqnarray}
\tilde{\mathcal{P}}_\nu(\varepsilon,\lambda):=\displaystyle -(J_1+2(\nu-1)J_3)\sum_{\tiny
\begin{array}[t]{c}
-2N\leq k\leq 2N \\ 
-N\leq \ell\leq N \\ 
|k+\ell|\leq N
\end{array}
}\hspace{-0.2cm}\frac{\gamma_{k+\ell}\gamma_\ell}{\varepsilon^2}e^{k\lambda\varepsilon}-\nonumber\\
\displaystyle J_5\sum_{k=-N}^N\frac{1}{\varepsilon}(\gamma_k-\gamma_{-k})e^{k\lambda\varepsilon}-(J_2+2(\nu-1)J_4).
\label{tildeJlambda}
\end{eqnarray}
Let us introduce the following subsets 
\begin{center}
$\tilde{\mathcal{Q}}_\nu:=\left\{\lambda\in\mathbb{C}/\det(\tilde{\mathcal{P}}_\nu(\varepsilon,\lambda))=0\right\}$, $\forall \nu\in\mathbb{R}$.
\end{center}

\begin{proposition}
We assume that 
\begin{eqnarray}
|\tilde{\mathcal{Q}}_n|=|\tilde{\mathcal{Q}}_0|=4Nd,~\tilde{\mathcal{Q}}_n\cap \tilde{\mathcal{Q}}_0=\emptyset\nonumber\\
\det(\tilde{\mathcal{P}}_n(\varepsilon,0))\neq 0\mbox{ and }\det(\tilde{\mathcal{P}}_0(\varepsilon,0))\neq 0.
\label{generalconditionsforQEPDEL}
\end{eqnarray}
Then there exists solutions $\tilde{\mathbf{x}}_s$ and $\tilde{\mathbf{x}}_j$ to (\ref{DELchoreoCTsumsimple}) and (\ref{DELchoreoCT}) respectively of the shape (\ref{psdf}) inside the interval $[t_0+2N\varepsilon,t_f-2N\varepsilon]$.
\label{prop3.2}
\end{proposition}

\begin{proof}
The two last conditions (\ref{generalconditionsforQEPDEL}) imply that $0\notin\tilde{\mathcal{Q}}_n$. If we set $\zeta=e^{\lambda\epsilon}$, we see that the quantity $\tilde{\mathcal{P}}_n(\varepsilon,\lambda)\zeta^{2N}$ is a polynomial w.r.t. $\zeta$ of degree $4N$. So the equation $\zeta^{2Nd}\det(\tilde{\mathcal{P}}_n(\varepsilon,\lambda))=0$ gives rise to a polynomial equation w.r.t. $\zeta$ of degree $4Nd$.\\
Computation of the l.h.s. of (\ref{DELchoreoCTsumsimple}) is performed by using (\ref{ourbox}) and \cite[Lemma 6.1]{RS2}. We find 
\begin{equation*}
\sum_{\tiny
\begin{array}[t]{c}
-2N\leq k\leq 2N \\ 
-N\leq \ell\leq N \\ 
|k+\ell|\leq N
\end{array}
}\hspace{-0.2cm}\frac{1}{\varepsilon^2}\gamma_{k+\ell}\gamma_\ell\chi_\ell(t)\chi_{-k}(t)(J_1+2(n-1)J_3)\tilde{\mathbf{x}}_s(t+k\varepsilon )+(J_2+2(n-1)J_4)\tilde{\mathbf{x}}_s(t)
\end{equation*}
\vspace{-0.3cm}
\begin{equation}
\hspace{1.7cm}+\sum_{k=-N}^N\hspace{-0.2cm}\chi_{-k}(t)\frac{1}{\varepsilon}(\gamma_k-\gamma_{-k})J_5\tilde{\mathbf{x}}_s(t+k\varepsilon )+n\Box_{-\varepsilon}1J_6+nJ_7=0.
\label{DELchoreoCTxsexplicit}
\end{equation}
When $t$ lies in the interval $[t_0+2N\varepsilon,t_f-2N\varepsilon]$, the various characteristic functions $\chi_k(t)$ occuring in (\ref{DELchoreoCTxsexplicit}) are equal to $1$. Next, we define for $t_j\in[t_0,t_f]$ the grid $\mathcal{G}_{t_j,\varepsilon}=\{t_j+m\varepsilon ,m\in\mathbb{N}\}\cap[t_0,t_f]$. So, both restrictions of $\tilde{\mathbf{x}}_s(t)$ and $\tilde{\mathbf{x}}_j(t)$ to $\mathcal{G}_{t_j,\varepsilon}$ are vector-valued sequences satisfying linear constant matricial recurrences. The classical theory of those systems \cite{L1,L2,TM} shows that, provided the characteristic equation admits a number of roots equal to the order of the recurrence, $\tilde{\mathbf{x}}_s(t)$ has the shape 
$\tilde{\mathbf{x}}_s(t)=\tilde{\mathbf{x}}_{s,0}+\sum_{\lambda}e^{\lambda t}\tilde{\mathbf{x}}_{s,\lambda}$, $\forall t\in\mathcal{G}_{t_j,\varepsilon}$,
for some vectors $\tilde{\mathbf{x}}_{s,\lambda}$ and $\tilde{\mathbf{x}}_{s,0}$. Here, the order of recurrence is equal to $4Nd$ and it is also equal to the number of roots of the characteristic equation which is $|\tilde{\mathcal{Q}}_n|$. So we may plug the previous formula into (\ref{DELchoreoCTxsexplicit}) and we find
\begin{equation}
\sum_{\lambda}e^{\lambda t}\tilde{\mathcal{P}}_n(\varepsilon,\lambda)\tilde{\mathbf{x}}_{s,\lambda}+\tilde{\mathcal{P}}_n(\varepsilon,0)\tilde{\mathbf{x}}_{s,0}=n\Box_{-\varepsilon}1J_6+nJ_7.
\label{xsinDEL}
\end{equation}
Because the values of the function $e^{\lambda t}$, on the grid $\mathcal{G}_{t_j,\varepsilon}$, are the numbers $e^{\lambda t_j}\zeta^m$ with $m=\frac{t-t_j}{\varepsilon}\in\mathbb{N}$, all the functions $e^{\lambda t}$ on this grid are linearly independent. Indeed, a linear relationship between these functions would give rise to a Vandermonde determinant w.r.t. to the associated distinct numbers $\zeta$. Therefore, every non-constant function of $t$ must vanish in (\ref{xsinDEL}), which means that the ``phases" $\lambda$ occuring in $\tilde{\mathbf{x}}_s(t)$ are exactly the roots $\alpha$ of $\tilde{\mathcal{Q}}_n$. By assumption, $\tilde{\mathcal{P}}_n(\varepsilon,0)$ is invertible and $\tilde{\mathcal{P}}_n(\varepsilon,\alpha)$ is singular. Thus, we may choose $\tilde{\mathbf{x}}_{s,0}=n\tilde{\mathcal{P}}_n(\varepsilon,0)^{-1}(J_7+\Box_{-\varepsilon}1J_6)$ and $\tilde{\mathbf{x}}_{s,\lambda}\in \ker(\tilde{\mathcal{P}}_n(\varepsilon,\lambda))$. Finally, we have determined the general solution $\tilde{\mathbf{x}}_s$ to (\ref{DELchoreoCTsumsimple}) on the grid $\mathcal{G}_{t_j,\varepsilon}$, namely
\begin{equation}
\tilde{\mathbf{x}}_s(t)=\tilde{\mathbf{x}}_{s,0}+\sum_{\alpha\in\tilde{Q}_n}e^{\alpha t}\tilde{\mathbf{x}}_{s,\alpha}.
\label{solgenxsDEL}
\end{equation} 

Now, let us deal with $\tilde{\mathbf{x}}_j(t)$. This function satisfies the following functional equation, which is similar to (\ref{DELchoreoCTxsexplicit})
\begin{center}
$\displaystyle -\sum_{\tiny
\begin{array}[t]{c}
-2N\leq k\leq 2N \\ 
-N\leq \ell\leq N \\ 
|k+\ell|\leq N
\end{array}
}\hspace{-0.2cm}\frac{1}{\varepsilon^2}\gamma_{k+\ell}\gamma_\ell\chi_\ell(t)\chi_{-k}(t)(J_1-2J_3)\tilde{\mathbf{x}}_j(t+k\varepsilon)$\\
$\displaystyle -(J_2-2J_4)\tilde{\mathbf{x}}_j(t)-J_5\sum_{k=-N}^N\hspace{-0.2cm}\chi_{-k}(t)\frac{1}{\varepsilon}(\gamma_k-\gamma_{-k})\tilde{\mathbf{x}}_j(t+k\varepsilon )=$
\end{center}
\begin{equation}
2J_3\Box_{-\varepsilon}\Box_\varepsilon\tilde{\mathbf{x}}_s(t)+2J_4\mathbf{x}_s(t)-\Box_{-\varepsilon}1J_6+J_7.
\label{DELchoreoCTxlexplicit}
\end{equation}
Let us construct a particular solution to (\ref{DELchoreoCTxlexplicit}) for $t\in\mathcal{G}_{t_j,\varepsilon}$. By using the previous expression for $\tilde{\mathbf{x}}_s(t)$, the r.h.s. of (\ref{DELchoreoCTxlexplicit}) may be rewritten as
\begin{center}
$\displaystyle \Box_{-\varepsilon}1J_6-J_7-2(\Box_{-\varepsilon}\Box_\varepsilon 1J_3+J_4)\tilde{\mathbf{x}}_{s,0}-2\sum_{\alpha\in\tilde{\mathcal{Q}}_n}e^{\alpha  t}(J_3+\theta_\alpha J_4)\tilde{\mathbf{x}}_{s,\alpha}$
\end{center}
where $\theta_\alpha=e^{-\alpha t}\Box_{-\varepsilon}\Box_\varepsilon e^{\alpha t}=\frac{1}{\varepsilon^2}\sum_{k,j}\gamma_{k+j}\gamma_je^{k\alpha\varepsilon}$. Now, if we substitute $\displaystyle \tilde{\mathbf{x}}_j(t)=\tilde{\mathbf{x}}_{j,0}+\sum_{\alpha\in\tilde{\mathcal{Q}}_n}e^{\alpha t}\tilde{\mathbf{x}}_{j,\alpha}$ in (\ref{DELchoreoCTxlexplicit}), we note that the l.h.s. of (\ref{DELchoreoCTxlexplicit}) is equal to
\begin{center}
$\displaystyle \tilde{\mathcal{P}}_0(\varepsilon,0)\tilde{\mathbf{x}}_{j,0}+\sum_{\alpha\in\tilde{\mathcal{Q}}_n}e^{\alpha t}\tilde{\mathcal{P}}_0(\varepsilon,\alpha)\tilde{\mathbf{x}}_{j,\alpha}$.
\end{center}
Because $\det(\tilde{\mathcal{P}}_0(\varepsilon,0))\neq 0$, we may define 
\begin{center}
$\displaystyle \tilde{\mathbf{x}}_{j,0}=\tilde{\mathcal{P}}_0(\varepsilon,0)^{-1}(2(\Box_{-\varepsilon}\Box_\varepsilon 1 J_3+J_4)\tilde{\mathbf{x}}_{s,0}-\Box_{-\varepsilon} 1 J_6+J_7)$.
\end{center}
Since $\tilde{\mathcal{Q}}_n\cap\tilde{\mathcal{Q}}_0=\emptyset$, the matrix $\tilde{\mathcal{P}}_0(\varepsilon,\alpha)$ is invertible for each $\alpha\in\tilde{\mathcal{Q}}_n$ and we may set 
\begin{center}
$\displaystyle\tilde{\mathbf{x}}_{j,\alpha}=2\tilde{\mathcal{P}}_0(\varepsilon,\alpha)^{-1}(J_3+\theta_\alpha J_4)\tilde{\mathbf{x}}_{s,\alpha}$.
\end{center}
Similarly to the case of C.E.L., we readily prove that $\tilde{\mathbf{x}}_{j,0}=\frac{1}{n}\tilde{\mathbf{x}}_{s,0}$ and $\tilde{\mathbf{x}}_{j,\alpha}=\frac{1}{n}\tilde{\mathbf{x}}_{s,\alpha}$, $\forall \alpha\in \tilde{\mathcal{Q}}_n$. At last, since $|\tilde{\mathcal{Q}}_0|=4Nd$, we conclude that the general solution to (\ref{DELchoreoCTxlexplicit}) on the grid $\mathcal{G}_{t_j,\varepsilon}$ is given by
\begin{equation}
\tilde{\mathbf{x}}_{j}(t)=\frac{1}{n}\tilde{\mathbf{x}}_s(t)+\sum_{\beta\in\tilde{\mathcal{Q}}_0}e^{\beta t}\tilde{\mathbf{x}}'_{j,\beta}
\label{solgenxlDEL}
\end{equation}
for some convenient vectors $\tilde{\mathbf{x}}'_{j,\beta}$ in $\ker(\tilde{\mathcal{P}}_0(\varepsilon,\beta))$.

If we drop the requirement that $t$ lies in $\mathcal{G}_{t_j,\varepsilon}$, i.e. if we remove the condition $\frac{t-t_j}{\varepsilon}\in\mathbb{N}$, the functions $t\mapsto \tilde{\mathbf{x}}_s(t)$ and $t\mapsto \tilde{\mathbf{x}}_j(t)$ may be extended by the preceding formulas to pseudo-periodic functions $t\mapsto \tilde{\mathbf{x}}_s(t)$ and $t\mapsto \tilde{\mathbf{x}}_j(t)$ respectively. Since the equations of motion are autonomous (independent w.r.t. $t$), these $n+1$ functions are solutions to (\ref{DELchoreoCTsumsimple}) and (\ref{DELchoreoCT}) respetively. Therefore, these functions are of the shape (\ref{psdf}) with $K=4Nd$ and $K=8Nd$ respectively and the proof is complete.
\end{proof}

\begin{remark}\rm
Solving D.E.L. with Dirichlet conditions leads to $n+1$ uncoupled linear systems, one of size $\sum_{\alpha \in \tilde{\mathcal{Q}}_n}\dim\ker(\tilde{ \mathcal{P}}_\nu(\varepsilon,\alpha))$ and the $n$ others of size $\sum_{\beta\in\tilde{ \mathcal{Q}}_0}\dim\ker(\tilde{\mathcal{P}}_0(\varepsilon,\beta))$. If those systems are Cramer, then the pseudo-periodic solution to D.E.L. exists and is unique.
\end{remark}

\section{Convergence issues}

Let us fix $\nu,\gamma_{-N},\ldots,\gamma _N,J_1,\ldots ,J_5$. Motivated by studying the convergence of the solutions to D.E.L. to the respective solutions to C.E.L., it is natural at first sight to ask if the matrix-valued function $\tilde{\mathcal{P}}_\nu(\varepsilon,\lambda)$ tends to $\mathcal{P}_\nu(\lambda)$ locally uniformly w.r.t. $\lambda\in\mathbb{C}$ as $\varepsilon$ tends to 0. Next, we recall the Hausdorff metric  
\begin{center}
$d_H(F_1,F_2):=\max \left\{\underset{x\in F_1}{\max}~\underset{y\in F_2}{\min}|x-y|,\underset{x\in F_2}{\max}~\underset{y\in F_1}{\min}|x-y|\right\}$,
\end{center}
defined for all nonempty finite subsets $F_1,F_2\subset\mathbb{C}$. Thus, we naturally investigate the convergence, in this sense, of 
\begin{center}
$\tilde{\mathcal{Q}}_\nu =(\det (\tilde{\mathcal{P}}_\nu(\varepsilon,.)))^{-1}\{0\}$ to $\mathcal{Q}_\nu =(\det(\mathcal{P}_\nu(.)))^{-1}\{0\}$
\end{center}
as $\varepsilon$ tends to 0. In order to prove this result, we shall need the following Theorem of Cucker and Corbalan \cite{CC}.
\begin{theorem}
Let $P(X)=a_0X^m+a_1X^{m-1}+\ldots+a_m\in\mathbb{C}[X]\backslash\{0\}$. Let $\xi_1,\ldots,\xi_r$ be its roots in $\mathbb{C}$, with multiplicities $\mu_1,\ldots,\mu_r$ respectively, and let $\mathcal{B}_1,\ldots,\mathcal{B}_r$ be disjoint disks centered at $\xi_1,\ldots,\xi_r$ with radii $\varepsilon_0$ and contained in the open disk centered at 0 with radius $1/\varepsilon_0$. Then, there is a $\delta\in\mathbb{R}^+$, such that, if $|b_j-a_j|<\delta$ for every $0\leq j\leq m$, then the polynomial $Q(X)=b_0X^m+b_1X^{m-1}+\ldots+b_m$ has $\mu_j$ roots (counted with multiplicity) in each $\mathcal{B}_j$ and $\deg(Q)-\deg(P)$ roots with absolute value greater than $1/\varepsilon_0$.
\label{thmCC}
\end{theorem}
It extends older results of Weber and Ostrowski to the case of perturbation of polynomials of distinct degrees. Hence, we must exclude the $4Nd-2d$ divergent roots, as $\varepsilon$ tends to 0, from the set $\tilde{\mathcal{Q}}_\nu$ to prove the second result of convergence mentioned above.
\begin{theorem}
We keep the assumptions (\ref{generalconditionsforQEPCEL}) and (\ref{generalconditionsforQEPDEL}) of Propositions \ref{prop3.1} and \ref{prop3.2}. We assume that $\Box _\varepsilon $ defined by (\ref{ourbox}) is such that $\gamma _{-N}\gamma _N\neq 0$ and 
\begin{equation}
\left\{\begin{array}{l}\Box _\varepsilon 1=0\\
\Box_\varepsilon t=1
\end{array}\right.,~\forall t\in
[t_0+2N\varepsilon ,t_f-2N\varepsilon ].  \label{condsurboxCV}
\end{equation}
Let $\nu\in\mathbb{R}$ and $K$ any compact neighbourhood of $\mathcal{Q}_\nu$. Then, when $\varepsilon$ tends to 0, 
$\tilde{\mathcal{P}}_\nu(\varepsilon,\lambda)$ tends to $\mathcal{P}_\nu (\lambda )$ locally uniformly in $\mathbb{C}$ and $\tilde{\mathcal{Q}}_\nu\cap K$ tends to $\mathcal{Q}_\nu$ in the Hausdorff sense.
\end{theorem}
\begin{proof} 
The assumptions (\ref{condsurboxCV}) are equivalent to the algebraic
equations 
\[
\sum_{k=-N}^N\gamma _k=0\mbox{ and }\frac 12\sum_{k=-N}^Nk(\gamma _k-\gamma_{-k})=1\label{conditions}
\]
since the characteristic functions are equal to 1 in $[t_0+2N\varepsilon
,t_f-2N\varepsilon ]$. In Theorem 6.1 of \cite{RS2}, we have proved that
these conditions are themselves equivalent to one or the other statements
\begin{itemize}
\item  for all $\mathbf{x}\in \mathcal{C}^2([t_0,t_f])$, $\displaystyle \lim_{\varepsilon \to 0}\Box _\varepsilon \mathbf{x}(t)=\dot {\mathbf{x}}(t)$ locally uniformly in $]t_0,t_f[$,
\item  for all $\mathbf{x}\in \mathcal{C}^2([t_0,t_f])$, $\displaystyle\lim_{\varepsilon \to 0}\Box _{-\varepsilon }\mathbf{x}(t)=-\dot {\mathbf{x}}(t)$ locally uniformly in $]t_0,t_f[$.
\end{itemize}
The mode of convergence means that for all $\delta>0$, $\Box_{\pm\varepsilon}\mathbf{x}(t)$ tends uniformly to $\pm\dot{\mathbf{x}}(t)$ in $[t_0+\delta,t_f-\delta]$ when $\varepsilon$ tends to 0. This convergence can not be improved since the functions $t\mapsto \Box_\varepsilon 1$ and $t\mapsto \Box_\varepsilon t$ are equal to 0 and 1 respectively only in the interval $[t_0+2N\varepsilon,t_f-2N\varepsilon]$. By composition of these properties we obtain
\begin{equation}
\displaystyle-\frac 1{e^{\lambda t}}\Box _{-\varepsilon }\Box _\varepsilon
e^{\lambda t}=-\sum_{\tiny
\begin{array}[t]{c}
-2N\leq k\leq 2N \\ 
-N\leq \ell \leq N \\ 
|k+\ell |\leq N
\end{array}
}\hspace{-0.2cm}\frac 1{\varepsilon ^2}\gamma _{k+\ell }\gamma _\ell \chi
_\ell (t)\chi _{-k}(t)e^{k\lambda \varepsilon }\underset{\varepsilon
\rightarrow 0}{\longrightarrow }\lambda ^2
\label{conv1}
\end{equation}
and
\begin{equation}
\displaystyle \frac{1}{e^{\lambda t}}\left(\Box_{\varepsilon}e^{\lambda
t}-\Box_{-\varepsilon}e^{\lambda t}\right)=\sum_{k=-N}^N\hspace{-0.2cm}%
\chi_{-k}(t)\frac{1}{\varepsilon}(\gamma_k-\gamma_{-k})e^{k\lambda%
\varepsilon}\underset{\varepsilon\rightarrow 0}{\longrightarrow}2\lambda.
\label{conv2}
\end{equation}
We see easily that the functions $e^{-\lambda t}\Box _{-\varepsilon}\Box _\varepsilon e^{\lambda t}$ and $e^{-\lambda t}(\Box _\varepsilon
e^{\lambda t}-\Box _{-\varepsilon }e^{\lambda t})$ defined at $\varepsilon =0$ by the respective values $-\lambda ^2$ and $2\lambda $ are continuous w.r.t. $\varepsilon $. The quantities in both sides in each equation are obviously the coefficients of $(J_1+2(\nu -1)J_3)$ and $J_5$ in (\ref{Jlambda}) and (\ref{tildeJlambda}) when $t$ lies in the interval $[t_0+2N\varepsilon,t_f-2N\varepsilon ]$. Hence, the mapping $\lambda \mapsto \tilde{\mathcal{P}}_\nu (\varepsilon ,\lambda )$ tends to $\lambda \mapsto \mathcal{P}_\nu (\lambda )$ uniformly on any compact subset of the preceding product, as $\varepsilon $ tends to 0.\\ 
Let us deal with $\tilde{\mathcal{Q}}_\nu$. We compute first
\begin{eqnarray}
\tilde{\mathcal{P}}_\nu(\varepsilon,\lambda)-\mathcal{P}_\nu(\lambda) = (J_1+2(\nu-1)J_3)(\lambda^2+e^{-\lambda t }\Box_{-\varepsilon}\Box_\varepsilon e^{\lambda t})+\nonumber\\
J_5(-2\lambda+e^{-\lambda t}\Box_{\varepsilon}e^{\lambda t}-e^{-\lambda t}\Box_{-\varepsilon}e^{\lambda t}),
\label{diffdet}
\end{eqnarray}
for all $\varepsilon\neq 0$, $t\in[t_0,t_f]$ and $\lambda \in\mathbb{C}$. The l.h.s. of (\ref{diffdet}) is independent of $t$ and the r.h.s. is  constant w.r.t. $t$ inside $[t_0+2N\varepsilon,t_f-2N\varepsilon]$, as we see from (\ref{conv1}) and (\ref{conv2}). Expanding in Taylor series the exponentials $e^{k\lambda\varepsilon}$ w.r.t. $\varepsilon$, we find a matrix-valued convergent Taylor series w.r.t. $\varepsilon$ for $\tilde{\mathcal{P}}_\nu(\varepsilon,\lambda)$. The coefficient of $\varepsilon^m$ in $\tilde{\mathcal{P}}_\nu(\varepsilon,\lambda)$ is a polynomial matrix w.r.t. $\lambda$, independent of $t$ inside $[t_0+2N\varepsilon,t_f-2N\varepsilon]$. Now, the determinant of such a convergent Taylor series is itself a convergent Taylor series. 

At this point we have established that $\zeta^{2Nd}\det(\tilde{\mathcal{P}}_\nu(\varepsilon,\lambda))$ is a polynomial of degree $4Nd$ w.r.t. $\zeta=e^{\lambda\varepsilon}$ and admits a Taylor expansion w.r.t. $\varepsilon$ starting at $\zeta^{2Nd}\det(\mathcal{P}_\nu(\lambda))$.\\
We choose $\varepsilon_0$ so that $\mathcal{Q}_\nu\subset\dot{K}\subset K\subset\overline{\mathcal{B}}(0,1/\varepsilon_0)$ and small enough to separate the elements of $\tilde{\mathcal{Q}}_\nu$. Let $\delta$ as in Theorem \ref{thmCC}. We choose next $\epsilon$ so that, if $1\leq m\leq 4Nd$, the  coefficient of $\zeta^m$ in $\det(\tilde{\mathcal{P}}_\nu)-\det(\mathcal{P}_\nu)$ is less than $\delta$. Now, we may formulate the conclusion of Theorem \ref{thmCC} as the following inclusion
\begin{center}
$\displaystyle \tilde{\mathcal{Q}}_\nu =(\det (\tilde{\mathcal{P}}_\nu(\varepsilon,.)))^{-1}\{0\}\subset\left(\mathbb{C}\backslash\mathcal{B}(0,1/\varepsilon_0)\right)\cup\left(\bigcup_{\lambda\in\mathcal{Q}_\nu}\mathcal{B}(\lambda,\varepsilon_0)\right)$.
\end{center}
As a consequence, intersecting both sides with $K$ we get $d_H(\tilde{\mathcal{Q}}_\nu\cap K,\mathcal{Q}_\nu)<\varepsilon_0$ for all $\varepsilon$ small enough. This ends the proof.
\end{proof} 

\begin{remark}\rm 
The convergence of $\tilde{\mathbf{x}}(t)$ to $\mathbf{x}(t)$ as $\varepsilon$ tends to 0 implies more complicated issues. Indeed, not only the phases $\tilde{\mathcal{Q}}_n$ and $\tilde{\mathcal{Q}}_0$ have to tend to $\mathcal{Q}_n$ and $\mathcal{Q}_0$ respectively but the amplitudes $\tilde{\mathbf{x}}_{s,\alpha}$, $\tilde{\mathbf{x}}_{j,\alpha}$, and $\tilde{\mathbf{x}}'_{j,\beta}$, where $\alpha\in\tilde{\mathcal{Q}}_n$ and $\beta\in\tilde{\mathcal{Q}}_0$, have to tend also to the respective amplitudes ${\mathbf{x}}_{s,\alpha}$, ${\mathbf{x}}_{j,\alpha}$, and ${\mathbf{x}}'_{j,\beta}$, where $\alpha\in{Q}_n$ and $\beta\in{Q}_0$. We refer to \cite{RS2} for an examination of the difficulties in the case $n=1$. 
\label{remark3.4}
\end{remark}

\section{Periodicity and choreographies}

We focus in this section on periodic and choreographic solutions. Let us define a choreography of $n$ particles $(\mathbf{x}_1,\ldots,\mathbf{x}_n)$ in $\mathbb{R}^d$ as a $T$-periodic solution to the equations of motion in which the trajectories differ one to the other by some delay of the shape $\displaystyle \frac{kT}{n}$, $k\in\{1,\ldots,n\}$. In other words, a choreographic solution is a $\mathcal{C}^2$ mapping $\mathbf{u}:\mathbb{R}/T\mathbb{Z}\rightarrow \mathbb{R}^d$ such that $\mathbf{u}(t+T)=\mathbf{u}(t)$ and such that the family $\{\mathbf{x}_j(t)\}_j$, defined by $\mathbf{x}_j(t)=\mathbf{u}(t+jT/n)$, satisfies for all $j$
\begin{center}
$\begin{array}{rcl}\ddot{\mathbf{u}}(t+\frac{jT}{n}) & = & \mathbf{F}_j(\mathbf{u}(t+\frac{T}{n}),\ldots,\mathbf{u}(t+T)),\\
-\Box_{-\varepsilon}\Box_{\varepsilon}\mathbf{u}(t+\frac{jT}{n}) & = & \tilde{\mathbf{F}}_j(\mathbf{u}(t+\frac{T}{n}),\ldots,\mathbf{u}(t+T)),
\end{array}$
\end{center}
i.e. the respective equations of motion C.E.L. and D.E.L. presented in (\ref{systdyn}).

\begin{theorem}
\begin{enumerate}
\item Under the assumptions of Proposition \ref{prop3.1}, all the solutions to C.E.L. are periodic if and only if  
\begin{eqnarray}
\mathcal{Q}_n\cup \mathcal{Q}_0\subset i\mathbb{R},~\forall \lambda',\lambda''\in \mathcal{Q}_n\cup \mathcal{Q}_0,~\lambda'/\lambda''\in\mathbb{Q},\label{commensurableCEL}
\end{eqnarray}
\item Under the assumptions of Proposition \ref{prop3.2}, all the pseudo-periodic solutions to D.E.L. are periodic if and only if  
\begin{eqnarray}
\tilde{\mathcal{Q}}_n\cup \tilde{\mathcal{Q}}_0\subset i\mathbb{R},~\forall \lambda',\lambda''\in \tilde{\mathcal{Q}}_n\cup \tilde{\mathcal{Q}}_0,~\lambda'/\lambda''\in\mathbb{Q}.\label{commensurableDEL}
\end{eqnarray}
\item If $\det(\mathcal{P}_n(0))\neq 0$, $\det(\tilde{\mathcal{P}}_n(\varepsilon,0))\neq 0$ and
\begin{eqnarray}
|\mathcal{Q}_0|=2d,~\mathcal{Q}_0\subset i\mathbb{R}^\star,~\forall \lambda',\lambda''\in \mathcal{Q}_0,~\lambda'/\lambda''\in\mathbb{Q},\label{commensurableCELchoreo}\\
|\tilde{\mathcal{Q}}_0|=4Nd,~\tilde{\mathcal{Q}}_0\subset i\mathbb{R}^\star,~
\forall \lambda',\lambda''\in \tilde{\mathcal{Q}}_0,~\lambda'/\lambda''\in\mathbb{Q},\label{commensurableDELchoreo}
\end{eqnarray}
then there exists choreographic solutions $\mathbf{x}_j(t)$ and $\tilde{\mathbf{x}}_j(t)$ to C.E.L. and D.E.L.. 
\end{enumerate}
\end{theorem}

\begin{proof}
\begin{enumerate}
\item We first notice that if $\{\mathbf{u}_{\ell}\}_{\ell=1,\ldots,K}$, is a family of nonzero vectors in $\mathbb{C}^d$, then the various functions $t\mapsto e^{\lambda_\ell t}\mathbf{u}_{\ell}$ are linearly independent iff the $\lambda_\ell$ are pairwise distinct. It relies on the nonsingularity of the Vandermonde matrix $\mathcal{V}(\lambda_1,\ldots,\lambda_K)$.
As a consequence, the function $\mathbf{u}(t)=\sum_{\ell=1}^Ke^{\lambda_\ell t}\mathbf{u}_{\ell}$ is periodic iff for some $T>0$ we have $\forall \ell$, $\lambda_\ell T\in 2i\pi\mathbb{Z}$ and this is equivalent to the requirement $\forall\ell$, $i\lambda_\ell\in\mathbb{R}^\star$ and $\forall j,k$, $\lambda_j/\lambda_k\in\mathbb{Q}^\star$. Therefore, the period $T$ of $\mathbf{u}(t)$ is $\inf\{T>0,\frac{T\lambda_\ell}{2i\pi}\in\mathbb{Z},\forall\ell\}$. Taking in account that the vectors $\mathbf{x}_{s,\alpha}$ and $\mathbf{x}'_{j,\beta}$, occuring in the proof of Proposition \ref{prop3.1}, may be chosen arbitrarily in the respective appropriate null spaces $\ker(\mathcal{P}_n(\alpha))$ and $\ker(\mathcal{P}_0(\beta))$, the previous properties of periodicity apply to the set of solutions to C.E.L. and give formula (\ref{commensurableCEL}). 
\item Pseudo-periodic solutions to D.E.L. are of the shape (\ref{psdf}) by using Proposition \ref{prop3.2} and the previous arguments apply.
\item We first observe that for each choreographic solution of the shape (\ref{psdf}), $\mathbf{x}_s(t)$ is necessarily constant. Indeed, 
\begin{equation}
\mathbf{x}_s(t)=\sum_{j=1}^n\mathbf{u}\left(t+\frac{jT}{n}\right)=n\mathbf{u}_0+\sum_{\ell=1}^K\mathbf{u}_{\ell}e^{\lambda_\ell t}\sum_{j=1}^n\left(e^{\frac{\lambda_\ell T}{n}}\right)^j.
\label{choreosol}
\end{equation}
By periodicity, we have $\lambda_\ell T\in 2i\pi\mathbb{Z}$ for all $\ell$ so that $\sum_{j=1}^n\left(e^{\frac{\lambda_\ell T}{n}}\right)^j=0$. Having this fact in mind, we may solve (\ref{CELchoreoCT}). Our assumptions imply that the solution $\mathbf{x}_j(t)$ to (\ref{CELchoreoCT}) may be written as (\ref{psdf}) with $K=2d$, since the underlying Quadratic Eigenvalue Problem satisfies $|\mathcal{Q}_0|=2d$ and $\det(\mathcal{P}_n(0))\neq 0$ (see Proposition \ref{prop3.1}). Plugging $\mathbf{x}_j(t)=\mathbf{u}(t+\frac{jT}{n})$ into (\ref{CELchoreoCT}) and using the linear independence of the summands (\ref{psdf}), we see that (\ref{CELchoreoCT}) is satisfied if and only if $\mathbf{u}_0=\mathcal{P}_n(0)^{-1}J_7$ and for all $\lambda\in \mathcal{Q}_0$, $\mathbf{u}_\lambda\in\ker(\mathcal{P}_0(\lambda))$. If we choose the vectors $\mathbf{x}_j(t_0)$ and $\dot{\mathbf{x}}_j(t_0)$ or $\mathbf{x}_j(t_0)$ and $\mathbf{x}_j(t_f)$ for all $j$ according to the preceding explicit form for $\mathbf{x}_j(t)$, we have justified the existence of choreographic solutions to C.E.L. with $\mathbf{x}_s=cst$. 

Let us deal now with D.E.L.. As seen in the proof of Proposition \ref{prop3.2}, each solution $\tilde{\mathbf{x}}_j(t)$ to (\ref{DELchoreoCT}) has the shape (\ref{psdf}) with $K=4Nd$ due to our assumptions on $\tilde{\mathcal{Q}}_0$ and $\tilde{\mathcal{P}}_n(\varepsilon,0)$. The remainder of the proof is entirely similar to C.E.L.. First, (\ref{DELchoreoCT}) is satisfied if and only if $\mathbf{u}_0=\tilde{\mathcal{P}}_n(\varepsilon,0)^{-1}(J_7-\Box_{-\varepsilon}J_6)$ and for all $\lambda\in\tilde{\mathcal{Q}}_0$, $\mathbf{u}_\lambda\in\ker(\tilde{\mathcal{P}}_0(\varepsilon,\lambda))$. Second, convenient choice of initial or boundary conditions guarantee the existence of choreographic solutions to D.E.L. with $\mathbf{x}_s=cst$.
\end{enumerate}
\end{proof}

\begin{remark}\rm
We may convert the existence of choreographic solutions into a linear algebra problem. Indeed, we add to the systems described at the end of the Sections 3.1 and 3.2 the following equations 
\begin{center}
$\mathbf{x}_{s,\alpha}=0$,  
$\mathbf{x}_{j,\alpha}=e^{\alpha(j-1)\frac{T}{n}}\mathbf{x}_{1,\alpha}$ and $\mathbf{x}'_{j,\beta}=e^{\beta(j-1)\frac{T}{n}}\mathbf{x}_{1,\beta}$,
\end{center}
for all $\alpha\in \mathcal{Q}_n, \beta\in \mathcal{Q}_0$ and $j=1,\ldots,n$. Due to (\ref{solgenxsCEL}) and (\ref{solgenxlCEL}) we see that, provided Dirichlet problem is well-posed, we find a choreographic solution.
\end{remark}

\begin{remark}\rm
In the litterature (see for instance \cite[Section 3.10]{TM}), the problem of the existence of choreographic solutions arises when one studies gyroscopic systems. The algebraic conditions P7 and P8 in \cite[Table 1.1]{TM} amount to require that $J_1=\tra{J}_1>0$, $J_2=\tra{J}_2>0$, $J_3$ and $J_4$ symmetric and small enough compared to $J_1$ and $J_2$ respectively. 
\end{remark}

 
\section{Numerical experiments on choreographies}

Experimental and working algorithms performed in this last section are implemented in Maple and Matlab. We deal with real symmetric matrices $J_1,\ldots,J_4$, zero vectors $J_6,J_7$, small dimension systems ($d=2,3$) since it displays already the main features, and arbitrary number of particles. Furthermore, existence of periodic or choreographic solutions requires that $J_5=0$. Let us give some details on the choice of the matrices $J_i$. Given $J_1,J_2,J_3$, we set, if $d=2$, $J_4=\frac{1}{2}J_2+\frac{1}{2}(J_1-2J_3)\begin{pmatrix}j_1 & j_2\\j_3 & j_4\end{pmatrix}$. Identifying the coefficients of the polynomial $\det(\mathcal{P}_0(\lambda))$ with those of $(\lambda^2+\beta_1^2)(\lambda^2+\beta_2^2)$ and requiring that $J_4=\tra{J}_4$, we get three equations on $j_1,j_2,j_3$, the coefficient $j_{4}$ standing free. Thus, we may choose $J_1+2(n-1)J_3$ and $J_1-2J_3$ definite positive, $J_2-2J_4$ and $J_2+2(n-1)J_4$ definite negative.

We present in Figures \ref{figure1} and \ref{figure2} the graphs of two solutions to gyroscopic C.E.L., sharing the same matrices $J_1,J_2,J_3$. 
On the left, a typical periodic curve obtained by considering $\beta_1=4i$ and $\beta_2=10i$ and on the right, a non-periodic curve. Incommensurability between $4i$ and $7i\sqrt{2}$ explains the non-choreographic behaviour of the curve, as mentioned in property (\ref{commensurableCEL}). 

\begin{figure}[!ht]
\begin{minipage}[c]{.46\linewidth}
\includegraphics[width=4cm,angle=270]{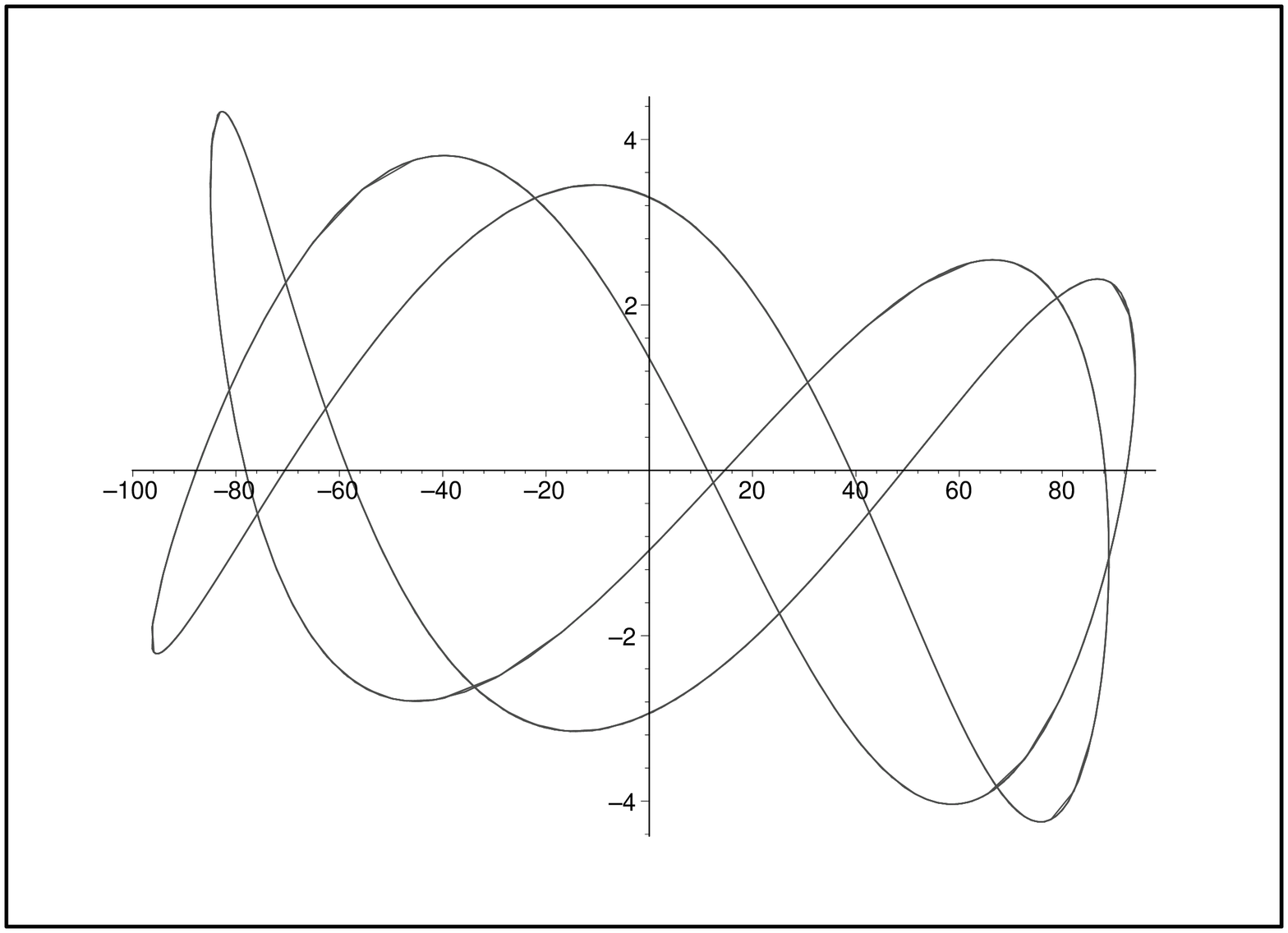}
\caption{C.E.L., $\beta_1=4i$ and $\beta_2=10i$}
\label{figure1}
\end{minipage} \hfill
\begin{minipage}[c]{.46\linewidth}
\includegraphics[width=4cm,angle=270]{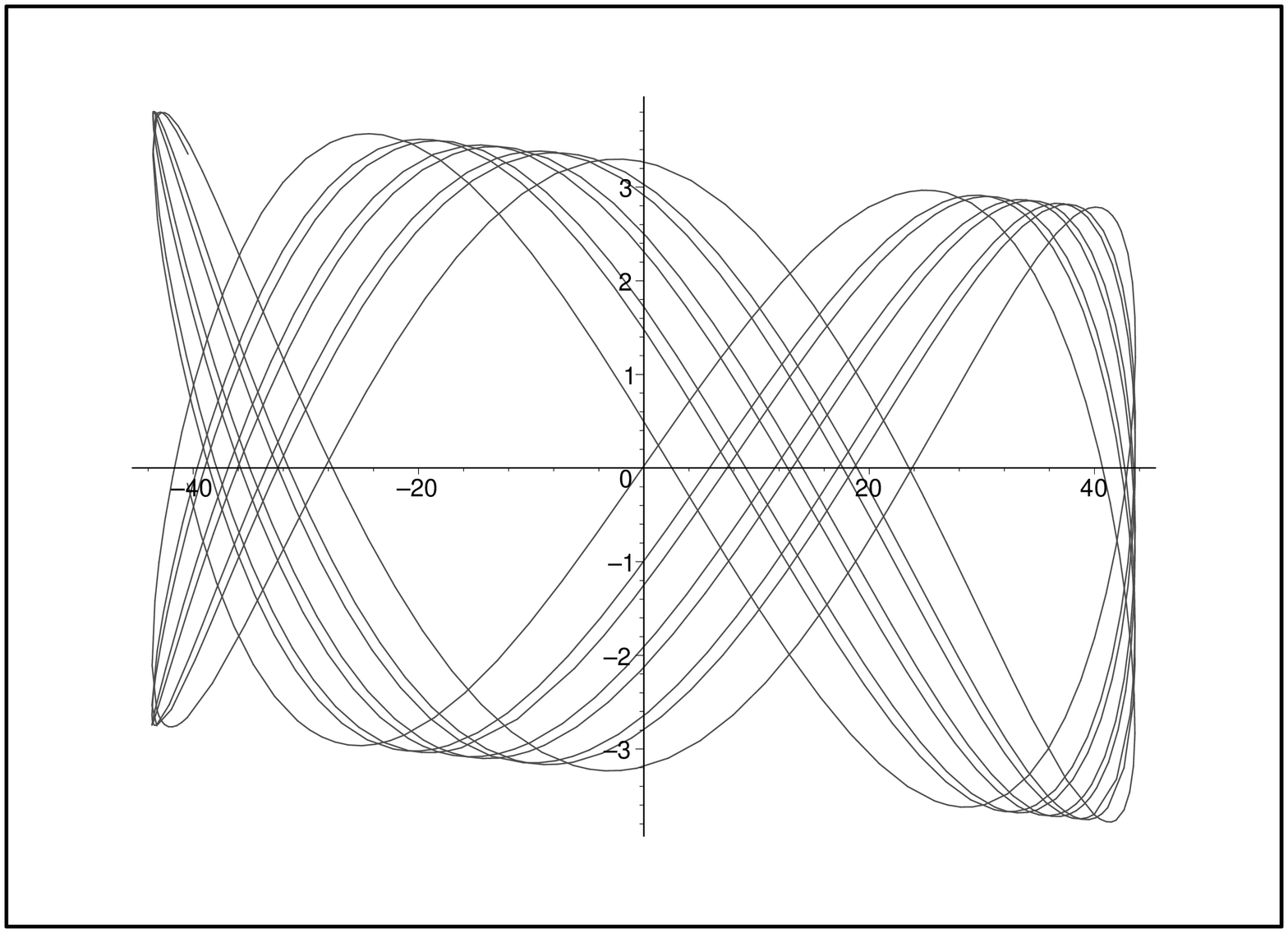}
\caption{C.E.L., $\beta_1=4i$ and $\beta_2=7i\sqrt{2}$}
\label{figure2}
\end{minipage}
\end{figure}

Let us deal now with D.E.L.. For sake of clarity, we shall denote by $\mathbf{x}_j(t)$ and $\mathbf{y}_{j,M}(t)$, $\forall j\in\{1,\ldots,n\}$, the unique solution to C.E.L (\ref{CELchoreoCT}) and the unique pseudo-periodic extension to $[t_0,t_f]$ of the unique solution to D.E.L. (\ref{DELchoreoCTxlexplicit}) on the grid $\mathcal{G}_{t_0,\varepsilon}$ with $\varepsilon=\frac{t_f-t_0}{M}$ respectively.\\
First, we give some hints to solve (\ref{DELchoreoCTxlexplicit}). When $t\in[t_0+2N\varepsilon,t_f-2N\varepsilon]$, we may compute $\mathbf{y}_j(t_0+2N\varepsilon)$ as a function of $\mathbf{y}_j(t_0+k\varepsilon)$ with $k$ varying from $-2N$ to $2N-1$. If $t\notin[t_0+2N\varepsilon,t_f-2N\varepsilon]$, some of the characteristic functions occuring in (\ref{DELchoreoCTxlexplicit}) vanish and solving (\ref{DELchoreoCTxlexplicit}) must be slightly modified, see more details in \cite{RS2}. We consider the matrices $J_1=\begin{pmatrix}7 & 2\\2 & 7\end{pmatrix}$, $J_2=\begin{pmatrix}5 & -1\\-1 & 5\end{pmatrix}$, $J_3=\begin{pmatrix}8 & 1\\1 & 8\end{pmatrix}$ and $(\beta_1,\beta_2)=(2i,5i)$. In a first experiment, we use an operator $\Box_{\varepsilon,k}$ such that $N=1$ and $(\gamma_{-1},\gamma_0,\gamma_1)=(-\frac{1}{2}+ik,-2ik,\frac{1}{2}+ik)$ where $k\in\mathbb{R}$. Because $\mathbf{x}_{j+1}(t)=\mathbf{x}_{j}(t+\frac{1}{n}T)$ and $\mathbf{y}_{j+1,M}(t)=\mathbf{y}_{j,M}(t+\frac{1}{n}\tilde{T})$, all the particles have the same trajectory, either in both cases C.E.L. and D.E.L.. Figure \ref{figure6} depicts the curves of $\mathbf{y}_{j,35}$, $\mathbf{y}_{j,42}$, $\mathbf{y}_{j,75}$ and $\mathbf{x}_{j}$. 
\begin{figure}[!ht]
\begin{minipage}[c]{.46\linewidth}
\includegraphics[width=5.7cm,height=3.7cm]{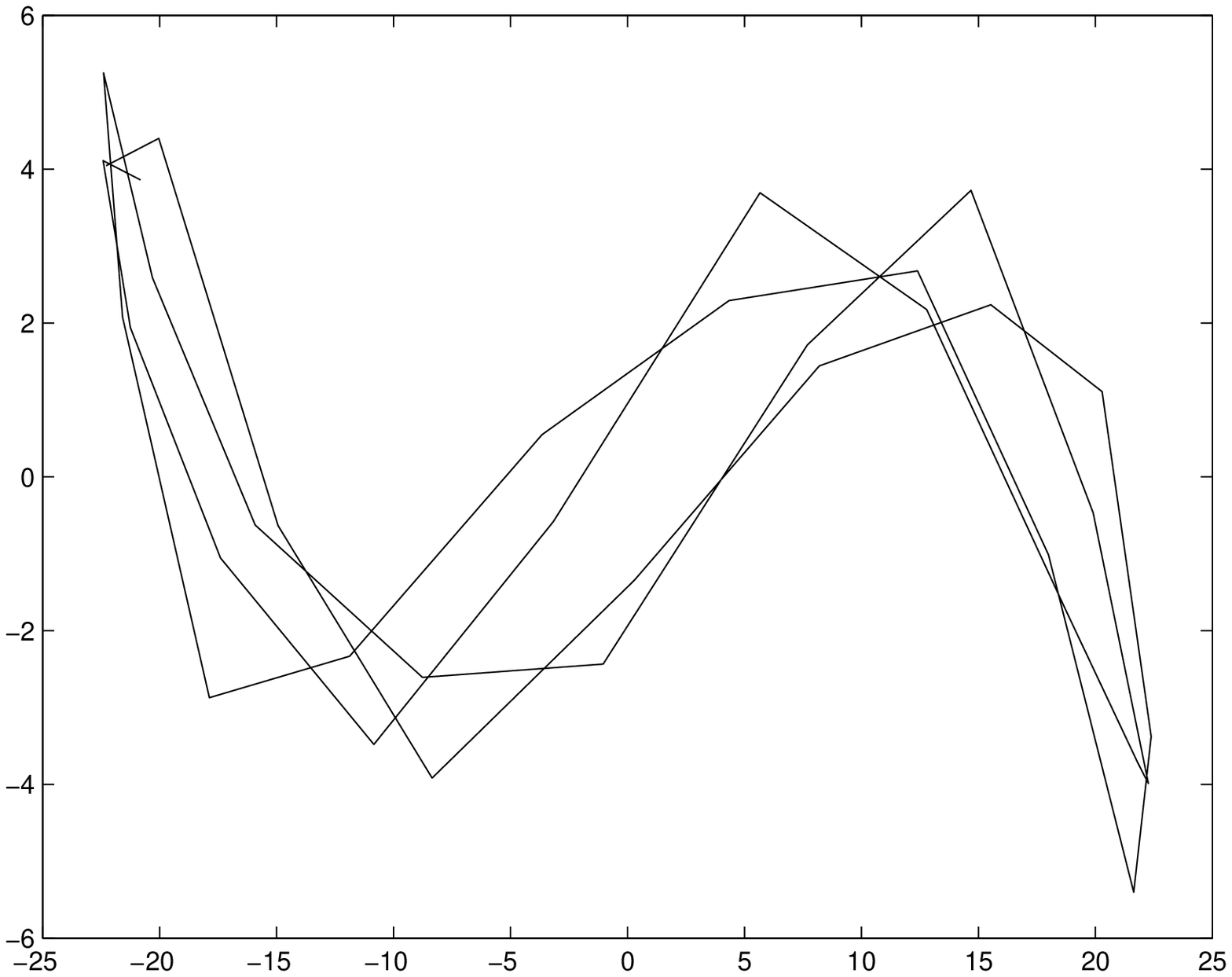}
\end{minipage} \hfill
\begin{minipage}[c]{.46\linewidth}
\includegraphics[width=5.7cm,height=3.7cm]{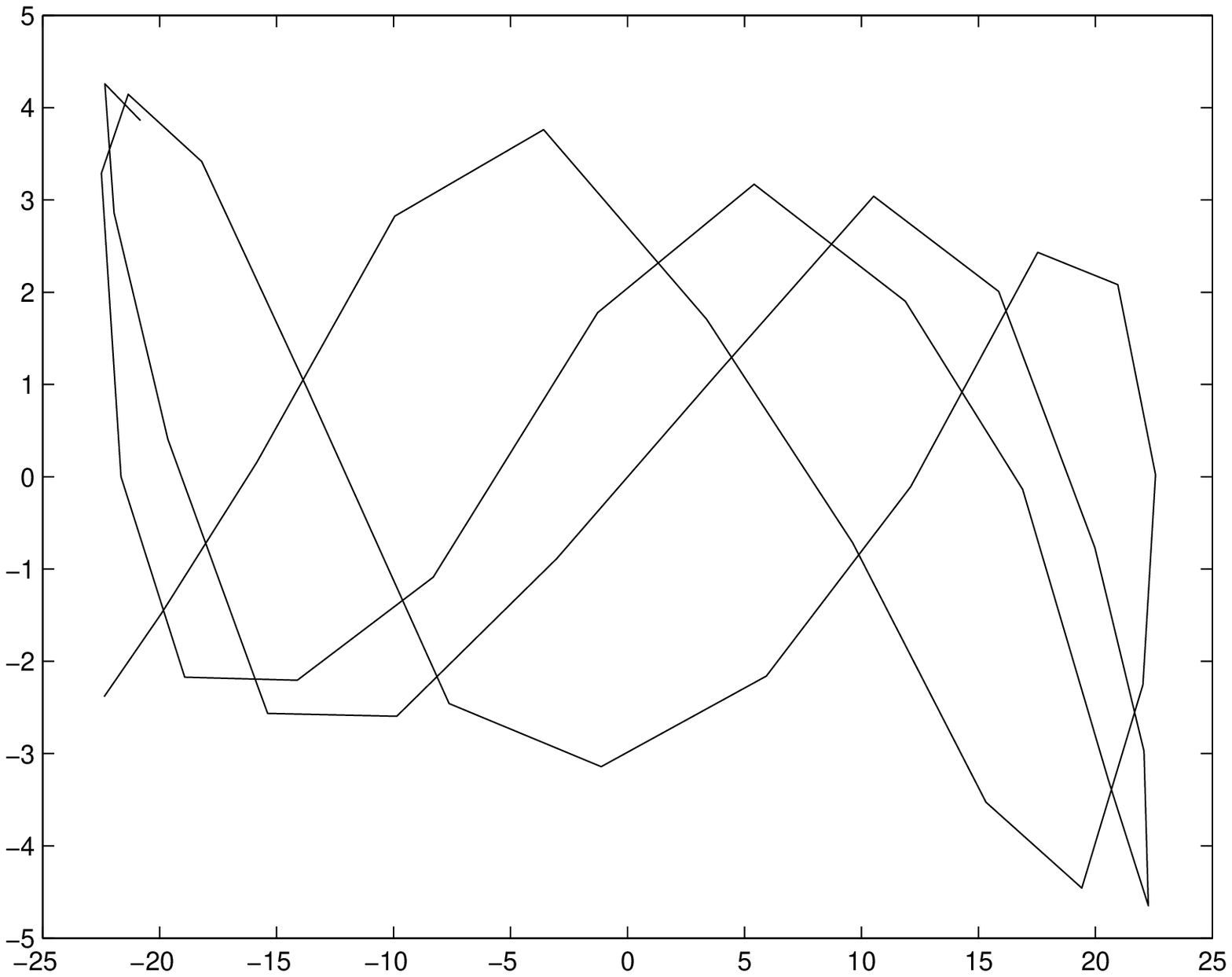}
\end{minipage}
\begin{minipage}[c]{.46\linewidth}
\includegraphics[width=5.7cm,height=3.7cm]{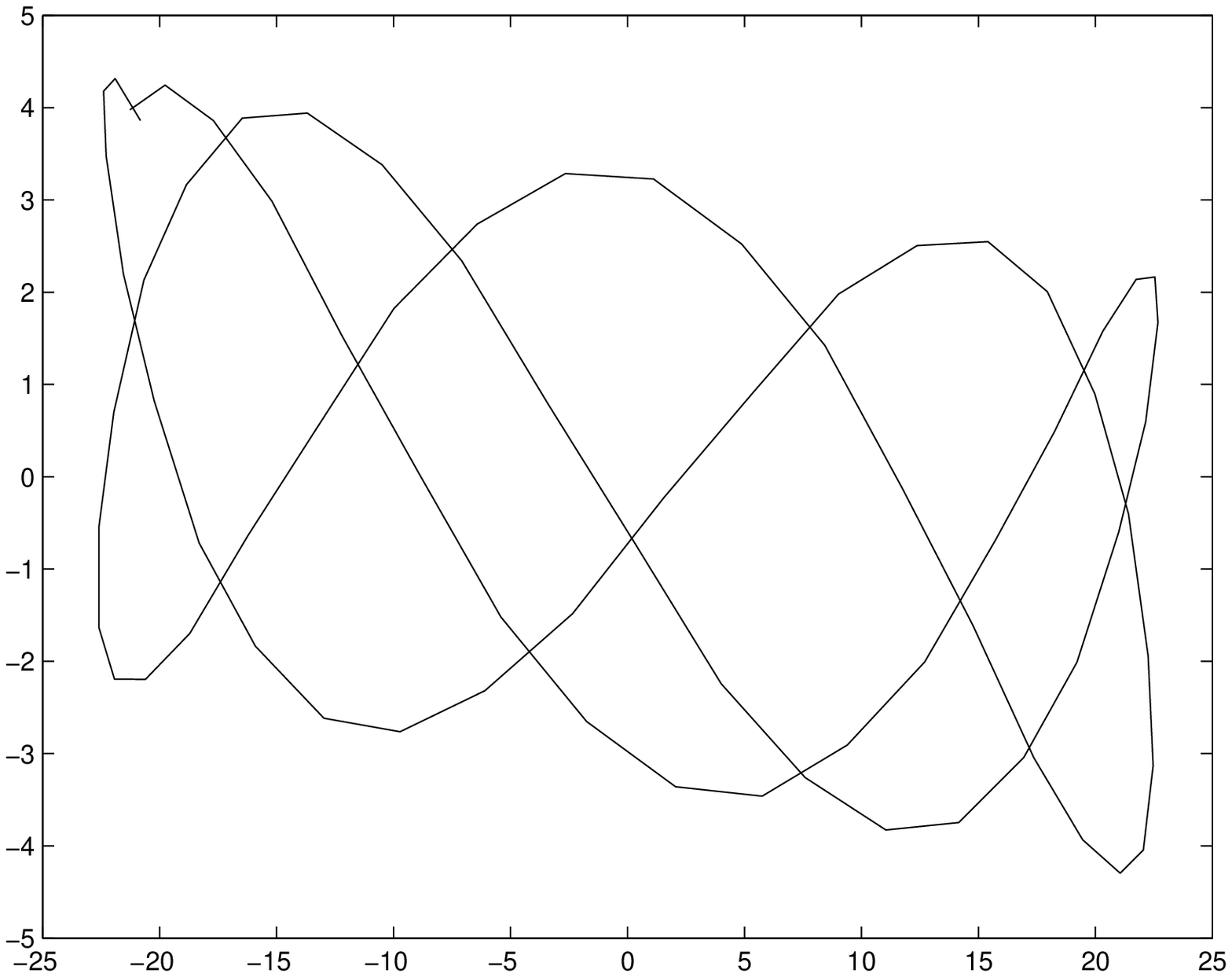}
\end{minipage}\hfill
\begin{minipage}[c]{.46\linewidth}
\includegraphics[width=5.7cm,height=3.7cm]{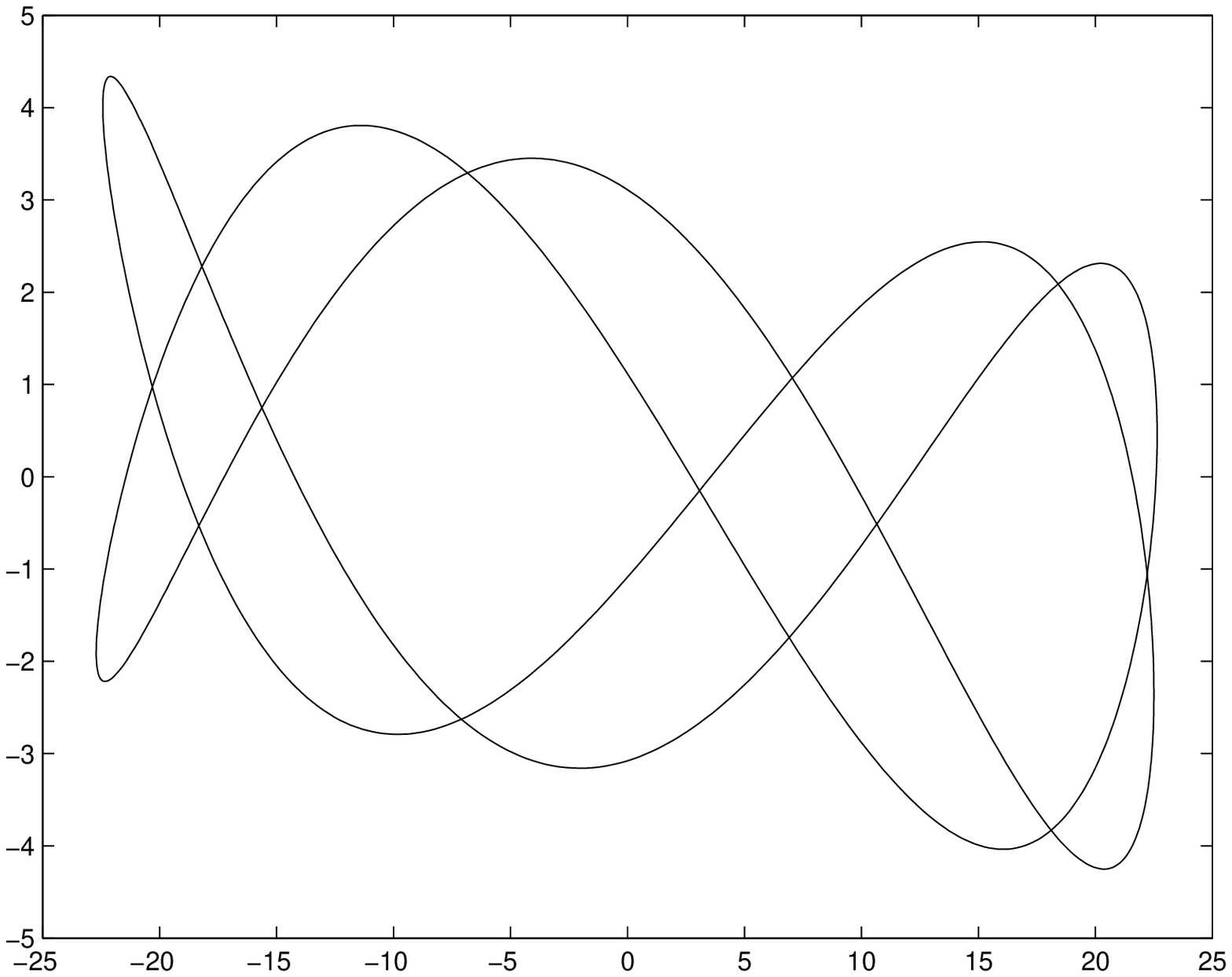}
\end{minipage}
\caption{Solution of D.E.L for $M=35,42,75$ and solution to C.E.L. in $\mathbb{R}^2$}
\label{figure6}
\end{figure}

Since our algorithms are suitable for each dimension, we provide also an example of quadratic choreography with $d=3$ in Figure \ref{figure10}. We choose $J_1,\ldots,J_4$ such that $\det(\tilde{\mathcal{P}}_0(\varepsilon,\lambda))=(\lambda^2+1)(\lambda^2+4)(\lambda^2+9)$ and the same operator $\Box_{\varepsilon,k}$ than previously.
\begin{figure}[!ht]
\begin{minipage}[c]{.46\linewidth}
{\includegraphics[width=5.7cm,height=3.7cm]{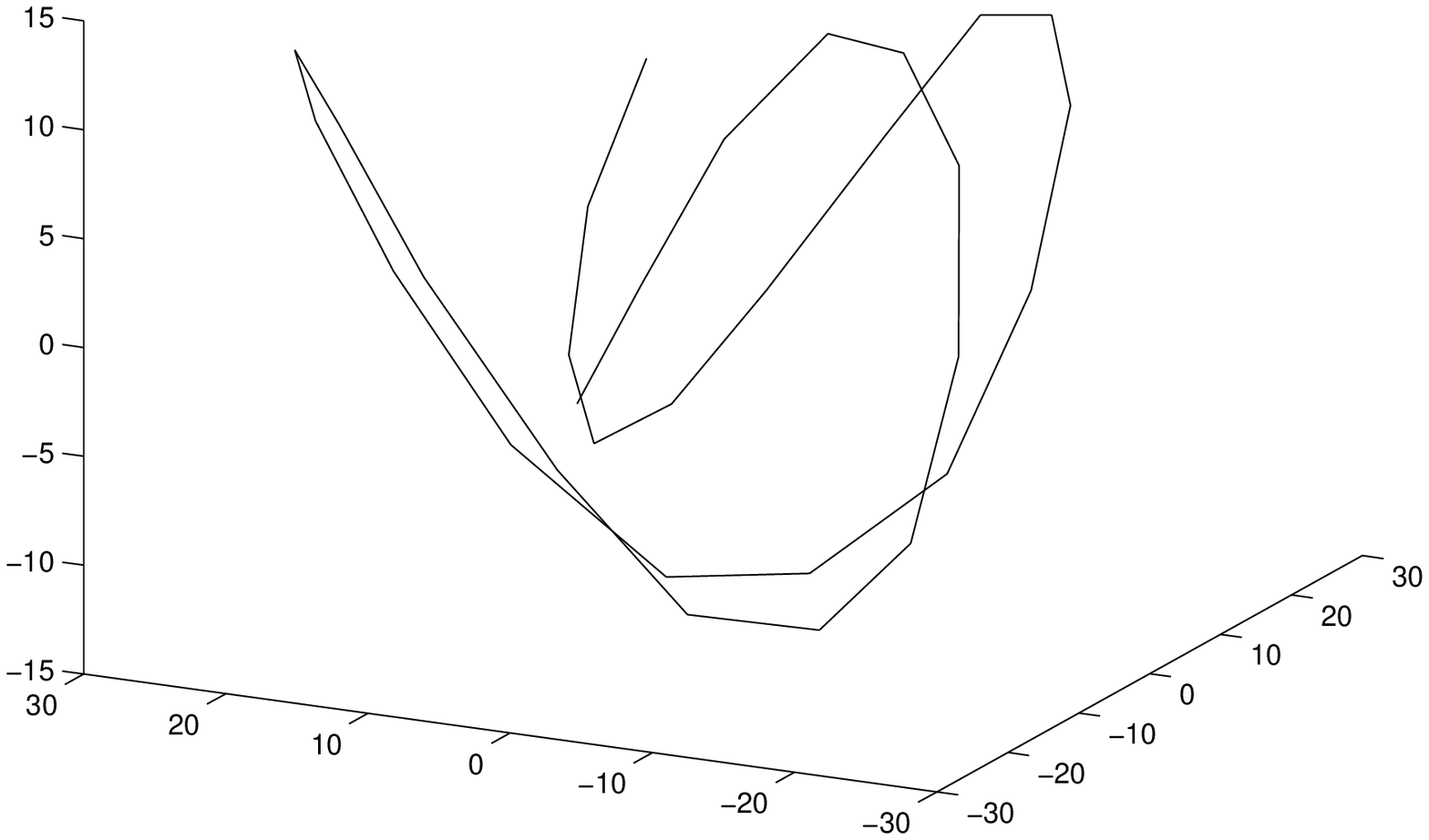}}
\end{minipage} \hfill
\begin{minipage}[c]{.46\linewidth}
\includegraphics[width=5.7cm,height=3.7cm]{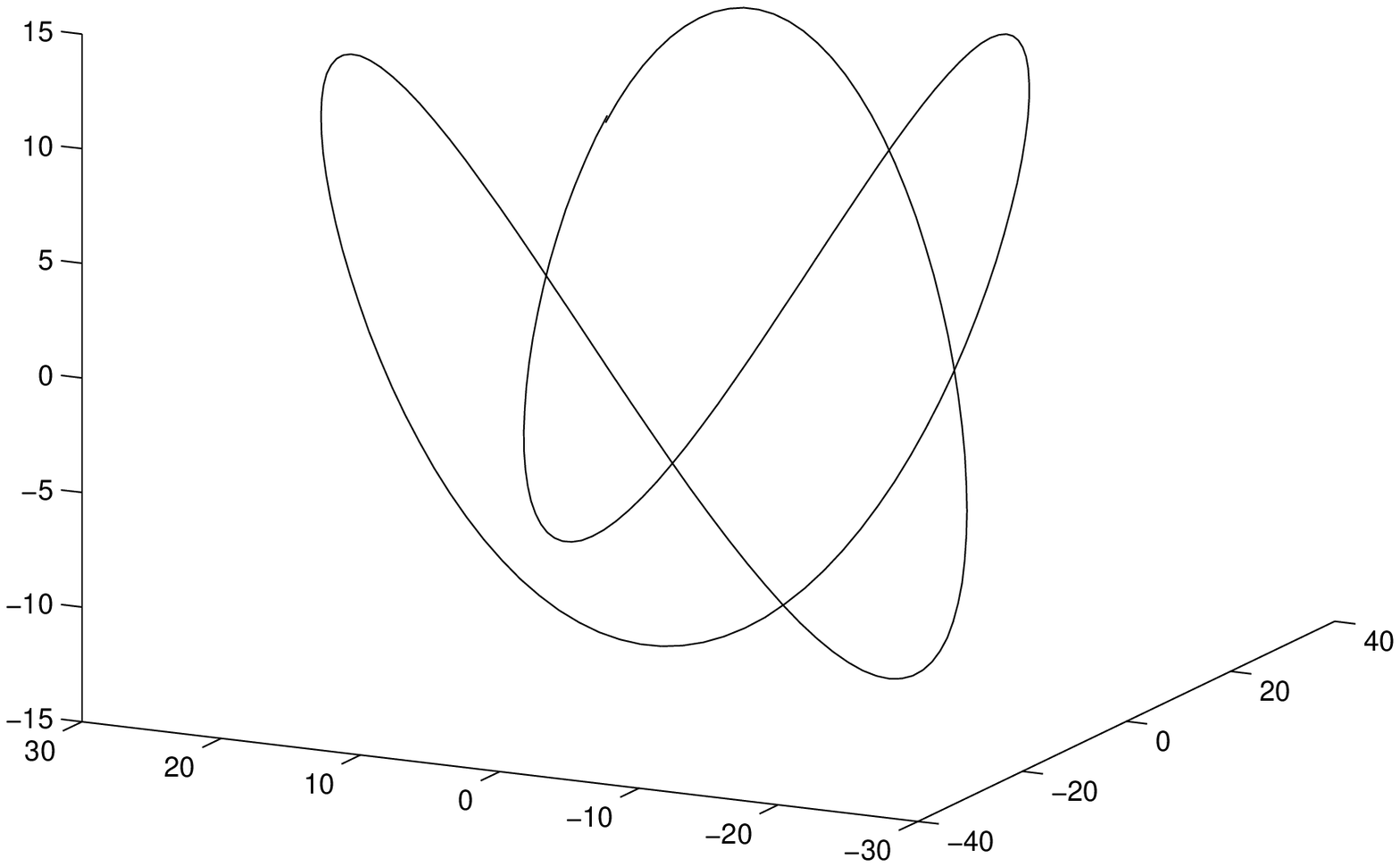}
\end{minipage}
\caption{Solution of D.E.L. for $M=30$ and solution to C.E.L. in $\mathbb{R}^3$}
\label{figure10}
\end{figure}

Figures \ref{figure6} and \ref{figure10} illustrate roughly the phenomenon of convergence, as $M$ increases, of the pseudo-periodic solution $\mathbf{y}_{j,M}(t)$ to D.E.L. to the solution $\mathbf{x}_j(t)$ to C.E.L., for all $j$, for all operator $\Box_{\varepsilon,k}$. In order that the approximation becomes meaningful, the number $M=\#(\mathcal{G}_{t_0,\varepsilon})-1$ must satisfy 
$$M^2\gg\frac{1}{|\gamma_{-1}\gamma_1|}(t_f-t_0)^2\rho((J_1-2J_3)^{-1}(J_2-2J_4))$$
as one sees from equation (\ref{DELchoreoCTxlexplicit}) ($\rho$ denotes here the spectral radius). Let us remark that this lower bound is independent on $d$ and still holds for general operators $\Box_\varepsilon$. 

In the next experiment, we still work with the previous lagrangian, with $N=1$ and with an operator $\Box_\varepsilon$ such that $\gamma_0=-(\gamma_{-1}+\gamma_1)$. We choose $M=100$ in order to avoid some erratic behaviour observed, for example, in the first plot ($M=35$) of Figure \ref{figure6}. We provide the plot of $-\log(\min(\|\mathbf{x}_j(t)-\mathbf{y}_j(t)\|_2,3M))$ as a function of $(\gamma_{-1},\gamma_1)\in\mathbb{R}^2$ in Figure \ref{figure7}. Two peaks occur at $(\gamma_{-1},\gamma_1)=\pm(1/2,1/2)$ and reveal a good approximation of $\mathbf{x}_j(t)$ by $\mathbf{y}_j(t)$. The two previous pairs are better understood if we have a look to the error $\mathbf{x}_j(t)-\mathbf{y}_j(t)$ with complex operators $\pm\Box_{\varepsilon,k}$. In Figure \ref{figure8}, we give the plot of the $2$-norm of the error with $(\gamma_{-1},\gamma_0,\gamma_{1})=(-\frac{1}{2}+ik,-2ik,\frac{1}{2}+ik)$ for several values of $M$. The operator $\Box_{\varepsilon,0}=0$ seems to be in any case the better choice.
\begin{figure}[!ht]
\begin{minipage}[c]{.46\linewidth}
\includegraphics[width=5.7cm]{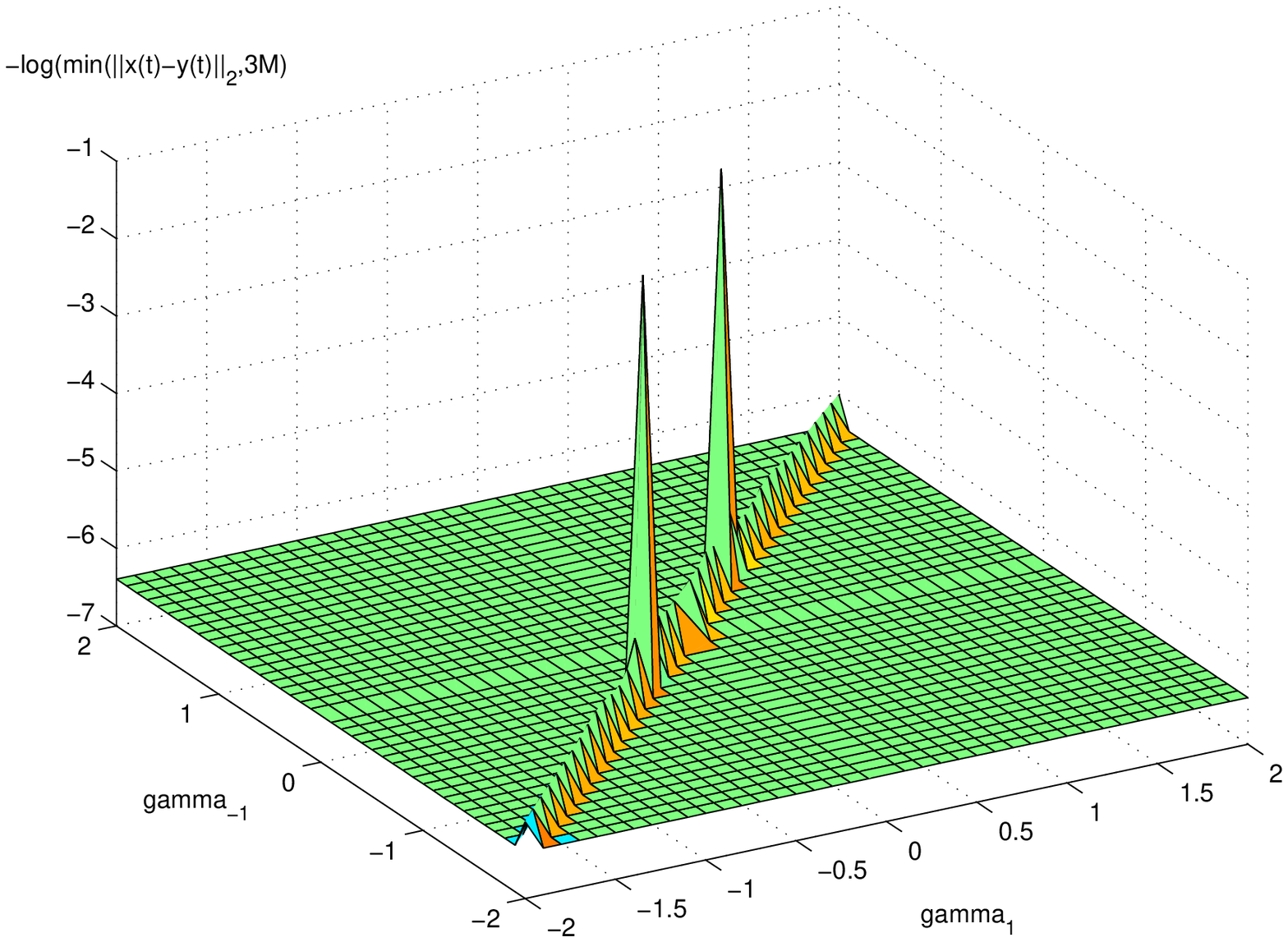}
\caption{$-\log(\min(\|\mathbf{x}_j(t)-\mathbf{y}_j(t)\|_2,3M))$ as function of $\gamma_{-1}$ and $\gamma_1$ ($M=100$)}
\label{figure7}
\end{minipage} \hfill
\begin{minipage}[c]{.46\linewidth}
\includegraphics[width=5.7cm]{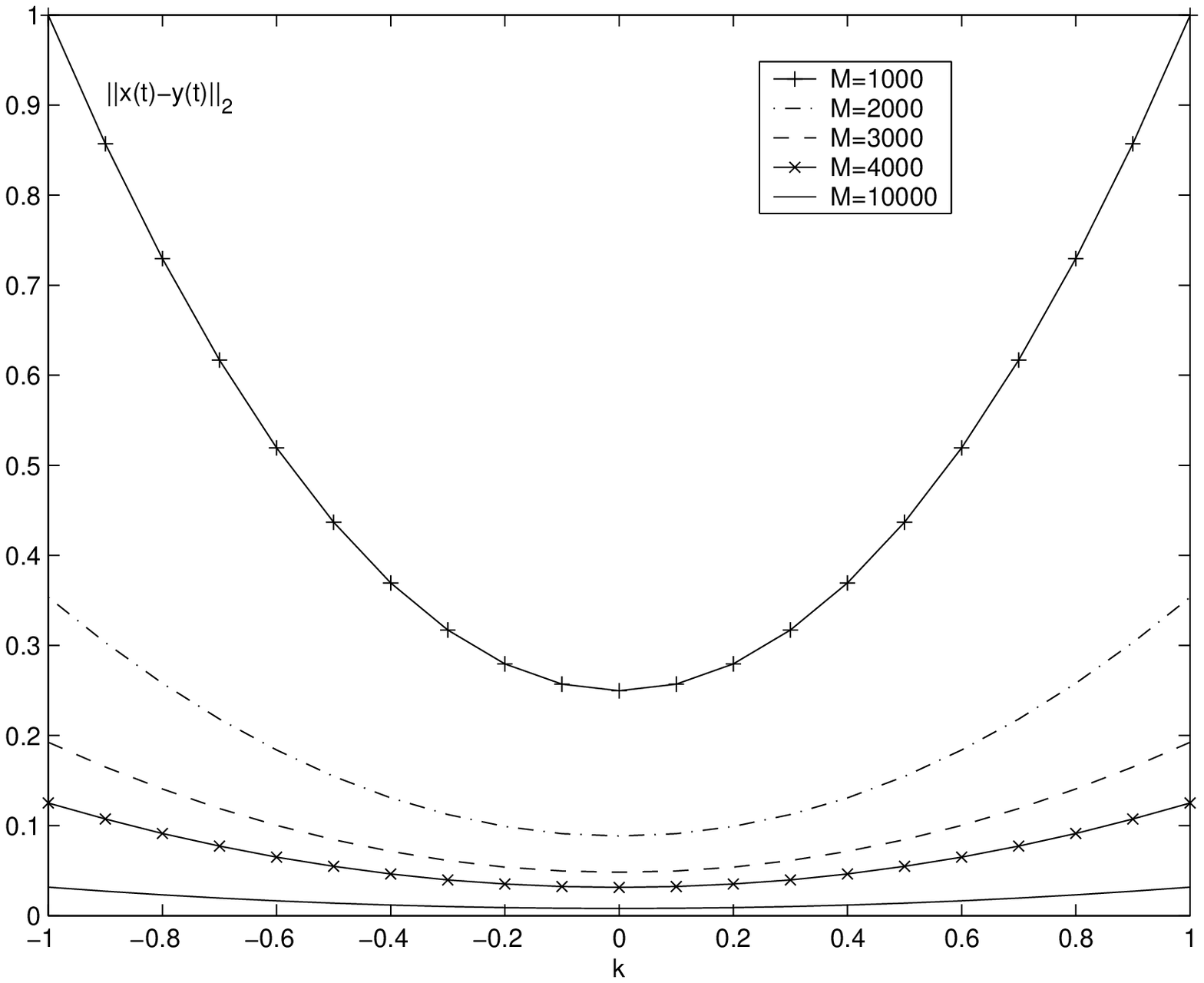}
\caption{$2$-norm of the error with $\Box_{\varepsilon,k}$ for several values of $M$}
\label{figure8}
\end{minipage}
\end{figure}

Let us conclude this paper with an additional remark on the convergence of solutions, completing Remark \ref{remark3.4}. Recall first that D.E.L. converges to C.E.L iff $\Box_\varepsilon$ if of the shape (\ref{ourbox}) and checks $\Box_\varepsilon 1=0$ and $\Box_\varepsilon t=1$, provided $t\in[t_0+2N\varepsilon,t_f-2N\varepsilon]$, see \cite[Definition 6.1. and Theorem 6.3]{RS1}. In that case, the condition $(\gamma_{-1},\gamma_0,\gamma_{1})=(-\frac{1}{2}+ik,-2ik,\frac{1}{2}+ik)$ is linked to the inclusion $\tilde{\mathcal{Q}}_0\subset i\mathbb{R}^\star$ as mentioned in \cite[Proposition 5.1]{RS2}) for the special case $d=1$. Based on the preceding experiments, we conjecture that under mild condition of non-resonance of the lagrangian, the solution to D.E.L. converges to the solution to C.E.L., as $\varepsilon$ tends to 0. 


\end{document}